\documentclass[10pt,twoside]{article}
\usepackage[pdfauthor={Benoit Charbonneau and Jacques Hurtubise},%
            pdftitle={Singular Hermitian-Einstein monopoles on the product of
a circle and a Riemann surface},%
pdffitwindow=false,%
colorlinks=true,linkcolor=black,urlcolor=black,citecolor=black]{hyperref}

\usepackage{fancyhdr}
\usepackage{amsmath}

\usepackage[margin=2cm]{geometry}

\fancypagestyle{plain}{%
   \fancyhead[LO]{} 
   \fancyhead[CO]{}\fancyhead[RO]{}
    \fancyfoot[LO]{}

}

\newenvironment{proof}{\topsep=\smallskipamount \partopsep=0pt  %
 \begin{trivlist} \itemindent=\parindent                        %
 \item[\hskip \labelsep\emph{Proof:}]}{\end{trivlist}}      %
\let\qed=\relax                                                 %
\def\qed                                                        %
 {{\unskip\nobreak\hfil\penalty50                               %
   \quad\hbox{}\nobreak\hfil $\Box$                     %
   \parfillskip=0pt \finalhyphendemerits=0 \par}}               %

\newtheorem{theorem}{Theorem}[section]

\newtheorem{corollary}[theorem]{Corollary}
\newtheorem{definition}[theorem]{Definition}

\newtheorem{lemma}[theorem]{Lemma}
\newtheorem{proposition}[theorem]{Proposition}
\newtheorem{remark}[theorem]{Remark}

\usepackage{amsmath,amssymb}
\usepackage[square]{natbib}

\def\citep#1#2{\cite[{#1}]{#2}}

\newcommand{\scp}[1]{\langle #1\rangle}  
\newcommand{\trk}[1]{{\mathrm{tr}(#1)}}  
\newcommand{\rank}{\mathrm{rank}}  
\newcommand{\dvol}{d\mu}           
\newcommand{\OO}{\mathcal{O}}   
\newcommand{\EE}{\mathcal{E}}  
\newcommand{\FF}{\mathcal{F}}  
\newcommand{\QQ}{\mathcal{Q}} 
\newcommand{\VV}{\mathcal{V}} 
\newcommand{\MM}{\mathcal{M}} 
\newcommand{\NN}{\mathcal{N}} 
\newcommand{\HH}{\mathcal{H}} 

\newcommand{\bbk}{\mathbf{K}}  
\newcommand{\bbp}{\mathbf{P}}  
\newcommand{\PP}{\mathbf{P}}  
\newcommand{\bbh}{\mathbf{H}}  
\newcommand{\bbi}{\mathbf{I}}  
\newcommand{\id}{\bbi}         
\newcommand{\cyl}{\mathbf{S}}  

\DeclareMathAlphabet{\mathdj}{U}{msb}{m}{n}  
\newcommand{\N}{\ensuremath{\mathdj {N}}} 
\newcommand{\Z}{\ensuremath{\mathdj {Z}}} 
\newcommand{\R}{\ensuremath{\mathdj {R}}} 
\newcommand{\C}{\ensuremath{\mathdj {C}}} 

\newcommand{\U}{\mathrm{U}}
\newcommand{\SU}{\mathrm{SU}}
\newcommand{\Gl}{\mathrm{Gl}}
\newcommand{\Sl}{\mathrm{Sl}}
\newcommand{\Deck}{\mathfrak{D}} 

\newcommand{\ad}{\mathrm{ad}}

\newcommand{\tr}{\mathrm{tr}}

\newcommand{\End}{\mathrm{End}}
\newcommand{\Aut}{\mathrm{Aut}}
\newcommand{\diag}{\mathrm{diag}}

\begin{document}
\title{Singular Hermitian--Einstein monopoles on the product of a circle and  
 a Riemann surface}
\author{Benoit Charbonneau and Jacques Hurtubise}
\date{October 19, 2009.}
\maketitle

\begin{abstract}
In this paper, the moduli space of singular unitary Hermitian--Einstein
monopoles on
the product of a circle and a Riemann surface is shown to correspond to a
moduli space of stable pairs on the Riemann surface.  These pairs consist of a
holomorphic vector bundle on the 
surface and a meromorphic automorphism of the bundle.  The singularities of this
automorphism correspond to the singularities of the singular monopole.  We
then consider the complex geometry of the moduli space; in particular, we
compute dimensions, both from the complex geometric and the gauge theoretic
point of view.\footnote{The authors
wishes to thank Tom
Mrowka and Mark Stern for
useful discussions, and Marco Gualtieri and Paul Norbury for comments on the
first two versions of the paper. The second author is supported by NSERC and FQRNT.
The authors can be reached respectively at 
[Math Dept, Duke University,
Box 90320, Durham, NC 27708-0320, USA,
benoit@alum.mit.edu], and  [Dept of Math
and Stat, McGill U., 805 Sherbrooke St. W, Montreal, Canada H3A
2K6, jacques.hurtubise@mcgill.ca].}
\end{abstract}

\section{Introduction} 
In a recent paper about the geometric Langlands program,  Kapustin and Witten \cite{geometricLanglands} expound the
idea that the moduli of singular monopoles on
the product of a Riemann surface $\Sigma$ with an interval should mediate the
Hecke transforms which play a part in the geometric Langlands correspondence. A
particular case is when the product of the transforms gives back the original bundle, and in
understanding this relationship, it is then natural to ask what one gets as a
monopole moduli space when  one closes the interval to a circle.

Another motivation to study monopole moduli spaces on such a product is obtained
by specialising to when the surface is a torus. One of the main tools for
understanding
monopoles, and more generally anti-self-duality, has been the Nahm transform
heuristic. It tells us that singular monopoles on $T^3=S^1\times T^2$ should
correspond to instantons on some ``dual'' $\R\times T^3$.  (While aspects
of this correspondence have been elucidated in \cite{benoitpaper}, the
correspondence is still not completely proven.) Thus, even when studying smooth
solutions to the anti-self-duality equations, one is led to consider
configurations with singularities.

The monopoles under consideration in this paper are solutions to 
a generalisation of 
the Bogomolny equation 
\[F_\nabla=*d_\nabla\phi\]
 linking the curvature $F_\nabla$ of a unitary
connection $\nabla$ on a Hermitian bundle $E$ over a Riemannian three-manifold $Y$ and an
skew-Hermitian endomorphism $\phi$ of $E$ called a Higgs field.  
The generalisation is the addition, in the special case 
where $Y=S^1\times \Sigma$, of a constant central term:
\begin{equation*} 
F_\nabla-iC\id_E\omega_\Sigma=*d_\nabla \phi. 
\end{equation*}

When $Y$ is compact,  global smooth solutions to the standard Bogomolny equation 
are quite trivial: the connection $\nabla$ must be flat and $\phi$ must be parallel. Over open manifolds, monopoles and their moduli  have been extensively
studied over the past twenty-five years, beginning with the case of $\R^3$. 
For $\R^3$, Hitchin in \cite{hitchinGeodesics} constructed all monopoles for 
$\SU(2)$ using the twistor methods of Penrose, Ward and Atiyah. Using this
construction, and
its natural extension to other gauge groups,  the moduli spaces were
described in
\cite{DonaldsonSU2,hurtubiseClassification,JarvisEuclideanMonopoles}. On
hyperbolic space, a similar description was 
given by Atiyah in 
\cite{Atiyah-monopoles-hyperbolic} (see also \cite{JarvisNorbury1997}). The
$\R^3$-spaces have natural metrics
 \cite{atiyah-hitchin};  for a unified discussion of Euclidean and hyperbolic
cases, see \cite{Nash-monopolemoduli}.  Monopoles on $\R^2\times S^1$ have been
studied in
\cite{cherkis2001} from the perspective of the Nahm transform.

The first study of singular monopoles (with Dirac-type singularities)
is due to Kronheimer in
\cite{kronheimerMSc} for Euclidean spaces. Pauly  computed in
\cite{pauly} the virtual dimension of the moduli space of singular
$\SU(2)$-monopoles and in \cite{Pauly-spherical-monopoles}   started the study
of singular monopoles on the round three-sphere. More
recently, Nash  considered in
\cite{Nash-singularhyperbolic} the twistor theory of singular
hyperbolic $\SU(2)$-monopoles. Norbury proved in 
\cite{Norbury-monopoles-boundary} the existence and uniqueness of singular
monopoles satisfying prescribed boundary conditions on an interval times a
surface.
The properties of certain moduli spaces of
singular
monopoles on $\R^3$ and $\R^2\times S^1$ allowed Cherkis--Kapustin
in \cite{cherkis1999,Cherkis-Kapustin-Super-QCD}
 and Cherkis--Hitchin in
\cite{cherkis-hitchin-Dk} to produce families of asymptotically locally flat
gravitational instantons.  

We shall
restrict our attention in this paper to the gauge group $\U(n)$. 
Our solutions have singularities, and we must
fix the nature of the singularity. Fortunately, in our case, there is a fairly
natural choice: we ask that near the singularity, the monopole, in essence,
should decompose into a sum of Dirac monopoles. These boundary conditions were
studied  by Pauly \cite{pauly}, who shows, by exploiting the geometry of the Hopf fibration, that there are natural
local lifts from $\R^3$ to $\R^4$ that tame the singularity.

The solutions to the Bogomolny equation we consider thus have Dirac-type 
singularities at fixed points $p_i= (t_i, z_i)\in S^1\times\Sigma$, 
where the $z_i$ are distinct. 
The bundle $E$ is defined on $S^1\times \Sigma\setminus\{p_1,\ldots, p_N\}$.
Precise definitions of the allowed singularities are given below; for the moment,
let $R_i$ be the geodesic distance, in $S^1\times \Sigma$, to $p_i$. 
On a small sphere surrounding the singularity $p_i$, the bundle is a sum of line bundles of degree  $k_{i1},\dots,k_{in}$ and near the singularity the monopoles will have Higgs field $\phi$ asymptotic to
$(\sqrt{-1}/2R_i)\diag(k_{i1},\dots,k_{in})$,
where $\vec{k}_i=(k_{i1},\ldots,k_{in})$ is a sequence of
integers, ordered so that $k_{i1}\geq \dots \geq k_{in}$. We collect these
sequences together as a sequence
\begin{equation}
\bbk=\bigl((\vec{k}_1,z_1),\ldots,(\vec{k}_N,z_N)\bigr).
\end{equation}
We also collect the length $T$ of the circle and the circle coordinates $t_i$ of the singularities   in a vector
\begin{equation}\vec{t}=(t_1,\ldots,t_N,T).\end{equation}

Denote by $E_t$ the restriction to $\{t\}\times\Sigma $ of the bundle $E$ for $t\neq t_i$. Let the degree of $E_0$ be  $k_0$. As one moves
through the point $p_i$, the degree of $E_t$ changes by $\trk{\vec{k}_i}:=\sum_j
k_{ij}$; in particular, it must be that $\sum_{i} \trk{\vec{k}_i}= 0$.

As with other examples of
moduli of  solutions to the anti-self-duality equations and their reductions, we
exploit the fact that the equations decompose into two components. The first
component simply states that we are dealing with a holomorphic object; the
second is variational in nature, and possesses a unique solution once one has a
solution to the first component, as long as the holomorphic object is stable in
a suitable sense. 
This general scheme often goes by the name of the Kobayashi--Hitchin correspondence.
In our case, we obtain the following theorem.
\begin{theorem}\label{thm:main}
The moduli space ${\MM}_{k_0}^{ir}(S^1\times \Sigma,
p_1,\ldots,p_N,\vec{k}_1,\ldots,\vec{k}_N)$ of $\U(n)$
irreducible Hermitian--Einstein monopoles on
$S^1\times\Sigma$ with $E_0$ of degree
$k_0$ and singularities at $p_j$ of type $\vec{k}_j$ maps bijectively to the
space ${\MM}_s(\Sigma,k_0, \bbk,\vec{t})$ of
$\vec{t}$-stable holomorphic pairs $(\EE,\rho)$ with 
\begin{itemize}
\item $\EE$ a holomorphic rank $n$ bundle of degree $k_0$ on $\Sigma$,
\item $\rho$ a meromorphic section of $\Aut(\EE)$ of the form $F_i(z)
\diag_j((z-z_i)^{k_{ij}})G_i(z)$ near $z_i$, with $F_i, G_i$ holomorphic and
invertible, and such that
$\det(\rho)$ has divisor $\sum_i\trk{  \vec{k}_i} z_i$.
\end{itemize}
More generally, 
the reducible HE-monopoles  correspond bijectively to 
 $\vec{t}$-polystable, but unstable, pairs.
\end{theorem}

The required notions, in particular of stability, are defined below.

Sections two through four of this paper are concerned with the proof of this theorem. 
The fifth section considers some examples. 
In the sixth and seventh sections, we use these ideas to consider   
monopoles on the product of an interval and a Riemann surface and 
on certain flat circle bundles over Riemann surfaces.

\section{Definitions}
\subsection{The Bogomolny  and  HE--Bogomolny equations}
Let $\Sigma$ be a Riemann surface, equipped with a Hermitian
metric; let $z$ denote a coordinate on $\Sigma$. Let $S^1$
be the circle, equipped with metric such that its circumference is $T$. Let
$t\in [0,T]$ be a coordinate on $S^1$, such that the $S^1$-metric is given by
$dt^2$.   We consider the three-fold $S^1\times \Sigma$ equipped with the
product metric, and as above, denote the submanifolds $\{t\}\times\Sigma$ by
$\Sigma_t$. Let $p_1,\ldots,p_N$ be a collection of points on $S^1\times\Sigma$;
set $p_i=(t_i, z_i).$ We suppose that the $z_i$ are distinct, and that the $t_i$ are ordered 
$t_1\leq\cdots\leq t_N$. We suppose the $t_i$
distinct and the origin of the circle chosen so that $t_i\neq 0$. The case where some $t_i$ are equal imposes no supplementary conceptual difficulty, but complicates the notation.

As above, we fix integers $k_{ij}, i=1,\ldots,N, j= 1,\ldots,n$ with
$\sum_{ij}k_{ij} = 0$. Now let $E$ be a Hermitian vector bundle of rank $n$ on 
$(S^1\times \Sigma) \setminus\{p_1,\ldots,p_N\}$, with degree $k_0$ on $\Sigma_0$,
and
degree $\trk{\vec{k}_i}=\sum_{j}k_{ij}$ on small spheres surrounding the $p_i$. Let
$E$ be equipped with a unitary connection $\nabla$ of curvature $F_\nabla$, and a
skew-Hermitian section $\phi$ of $\End(E)$ called the Higgs field. We say that
$(E,\nabla,\phi)$
satisfies the \emph{Bogomolny equation} if 
\begin{equation}
F_\nabla = *\nabla\phi.
\end{equation}
Let us recall how this equation is a reduction of the
anti-self-duality equation in four dimensions. Extend the circle to
a cylinder $\cyl=\R\times S^1$ with extra coordinate $s$, so that $w=s+it$ is a
holomorphic coordinate on $\cyl$. The product metric on $S^1\times\Sigma$
extends to a product K\"ahler metric $g$ on $\cyl\times\Sigma$ with K\"ahler
form $\Omega$. Let $\pi\colon \cyl\times\Sigma \to S^1\times\Sigma $ be the
projection, and let $\hat\nabla=(\frac{\partial}{\partial
s}+\pi^*\phi)ds+\pi^*\nabla$. The pair $(\nabla,\phi)$
satisfies the Bogomolny equation if and only if
$\hat\nabla$ satisfies the anti-self-dual (ASD) equation
$*F_{\hat\nabla}=-F_{\hat\nabla}$.  Another way to phrase this ASD equation
more in tune with the
complex structure is to split the curvature into bitype, and then to
isolate the component $\Lambda F_\nabla \cdot \Omega$ of the
curvature parallel
to the K\"ahler form $\Omega$. The
ASD equations become
\[F^{0,2}_{\hat\nabla}=0, F^{2,0}_{\hat\nabla}=0, \Lambda
F_{\hat\nabla}=0.\]
The first two equations simply state that one has a holomorphic object, with a
compatible Hermitian structure; the third, $\Lambda
F_{\hat\nabla}=0$, is, as we shall see, variational in nature.

On a compact complex surface, these equations impose constraints on the first
Chern class of the bundle.  There are
more general equations, the
Hermitian--Einstein
equations, that free us from this constraint.  They are
\begin{equation}
F_{\hat\nabla}^{0,2}= 0, F_{\hat\nabla}^{2,0}= 0, \Lambda F_{\hat\nabla} =iC\id,
\end{equation}
for  an imaginary constant multiple $iC$ of the identity endomorphism.
These equations can have non-trivial solutions on bundles
of arbitrary degree. 

In our case also, we find that the Bogomolny equations impose constraints, this time on the location of the singularities, that are too restrictive for our purposes. Just as the Bogomolny equations are reductions of the ASD equations, we have a reduced version of the Hermite--Einstein equation, which we call the
\emph{Hermitian--Einstein--Bogomolny equation} (or \emph{HE--Bogomolny} for short).
It is
\begin{equation}\label{eqn:HEBogomolny}
F_\nabla-iC\id_E\omega_\Sigma=*\nabla \phi
\end{equation}
If $ \nabla^{0,1}$ is the $0,1$ component of the covariant derivative along
$\Sigma$, and $F_\Sigma\cdot\omega$ is the component of the curvature along
$\Sigma$, and $\omega$ the K\"ahler form on $\Sigma$, these equations are
\begin{align}
[\nabla^{0,1}_\Sigma, \nabla_t-i\phi]&=0,\\
F_\Sigma - \nabla_t\phi &= iC\id_E.
\end{align}

\subsection{The Dirac monopole}\label{sec:Dirac}
Our singularities are modeled on those of the Dirac monopole. We begin by considering this example in some detail.
On $\R^3$, one has spherical coordinates related to the Euclidean coordinates by
$(t,x,y)=(R\cos\theta,R\cos\psi\sin\theta,R\sin\psi\sin\theta)$.
The volume form is
\[\dvol_{\R^3}=R^2\sin\theta dR\wedge d\theta\wedge d\psi=-R^2dR\wedge
d\bigl(\cos\theta d\psi\bigr).\]

Consider the Hermitian line 
bundle $L_k$ over $\R^3\setminus\{(0,0,0)\}$ defined by the transition function
$g_{\pi 0}=e^{ik\psi}$ on the complement of the $t$-axis, from 
$\theta\neq0$ to $\theta\neq \pi$.  
Hence any section $\sigma$ of $L_k$ is given by two functions 
\begin{align*}
\sigma_0&\colon \{(t,x,y)\in\R^3\setminus\{(0,0,0)\}
\mid \theta\neq 0\}\to\C,\text{ and}\\
\sigma_\pi&\colon \{(t,x,y)\in\R^3\setminus\{(0,0,0)\}
\mid \theta\neq \pi\}\to\C
\end{align*}
subject to the relation $\sigma_\pi=g_{\pi0}\sigma_0$.

Consider the connection $\nabla$ on $L_k$ defined by
the connection matrices
\begin{align}
  A_0&= \frac{ik}{2}(\phantom{-}1+\cos\theta)d\psi
\text{ on }\theta\neq0,\\
 A_\pi&=\frac{ik}{2}(-1+\cos\theta)d\psi
\text{ on }\theta\neq\pi.
\end{align}
 Note that while $d\psi$ is not defined at $\theta=0$ and $\theta=\pi$,
 the connection matrices $A_0$ and $A_\pi$ are smooth
respectively at $\theta=\pi$ and $\theta=0$.

Let
\begin{equation}
 \phi = \frac{ik}{2R}.
\end{equation}
We have
\[d\phi=-\frac{ik}{2R^2}dR=*\bigl(\frac{ik}2d(\cos\theta d\psi)\bigr)=*F_\nabla,\]
so $(\nabla,\phi)$ satisfy the Bogomolny equation.
  Notice that, restricted to
the sphere, we have
\[c_1(L_k)=\frac{i}{2\pi}\int_{S^2}F_\nabla=\frac{i}{2\pi}\bigl(-\frac{ik}
2\bigr)\int_{S^2}\dvol_{S^2}=k.\]
We call this special solution to the Bogomolny equation  the \emph{model
Dirac monopole of charge $k$}.

We now  work out the $\bar\partial$ operator 
$\nabla^{0,1}=\frac12(\nabla_x+i\nabla_y)$ 
and the change of gauge from a unitary to a 
``holomorphic gauge,'' 
that is a non-unitary trivialisation by a section $\sigma$
satisfying $\nabla^{0,1}\sigma=0, (\nabla_t-i\phi)\sigma = 0$.
Note that $\cos(\theta)=t/R$, and 
$d\psi=\frac{xdy-ydx}{r^2}$. Set $z = x+iy$. 
In the open set $\theta\neq0$,
\begin{equation*}
A_0 =\frac{ik}{2}\bigl(1+\frac tR\bigr)\frac{ xdy-ydx}{r^2}=
\frac{k}{4}\frac{1+t/R}{ r^2}(\bar zdz-zd\bar z)
\end{equation*}
and so \begin{equation*}
A^{0,1}_0 
=-\frac{k}{4}\frac{1+t/R}{ r^2}zd\bar z=-\frac{k}{4}\frac{(t
+\sqrt{t^2+z\bar z})}{z\bar z\sqrt{t^2 +z\bar z}}(zd\bar
z)\end{equation*}
Consider the radial (in $r$) form 
\begin{equation*}-\frac{k}{4}\frac{(t +\sqrt{t^2+r^2})}{r^2 \sqrt{t^2
+r^2}}(2rdr)\end{equation*}
with the same $(0,1)$ part. We can change the trivialisation  to
eliminate the $(0,1)$ part by applying a change of trivialisation
$g_0$ that solves, in cylindrical coordinates $t, z, \psi$, the equation
\begin{equation*}\label{ODEforg}\frac{\partial}{\partial r} \ln(g_0) =
-\frac{k}{2}\frac{(t +\sqrt{t^2+r^2})}{r\sqrt{t^2+r^2}}.\end{equation*}

A solution to this differential equation is
\begin{equation} g_0=(R-t)^{\frac{-k}{2}}. \end{equation}
In the new trivialisation, the connection form transformed by ($a\mapsto a - (dg)g^{-1}$) is given by $-\frac{k}{2R} dt + \frac{k(R+t)}{2r^2R} \bar z dz$.  Hence,
in this trivialisation $\nabla^{0,1}=\partial^{0,1}$, and  $\nabla_t-i\phi=\partial_t$. In this holomorphic gauge,  the metric is given by  
\begin{equation}\label{eqn:metricminus}
(g_0^*g_0)^{-1}=(R-t)^{k}.
\end{equation}
 
Similarly when $\theta\neq\pi$, after the change of trivialization,
\begin{equation} g_\pi=(R+t)^{\frac {k}{2}},\end{equation}
the connection form is $-\frac{k}{2R} dt - \frac{k(R-t)}{2r^2R} \bar z dz$.
In this trivialisation again $\nabla^{0,1}=\partial^{0,1}$ and
$\nabla_t-i\phi=\partial_t$. In this holomorphic gauge, the metric is given by 
\begin{equation}\label{eqn:metricplus}
(g_\pi^*g_\pi)^{-1}=(R+t)^{-k}. \end{equation}

The two new trivialisations are related by
\begin{equation} \label{eqn:zk} g_\pi g_{\pi0}g_0^{-1}=z^k. \end{equation}

The operator $(\nabla_t-i\phi)$ appears already in the work 
of Hitchin \cite{hitchinGeodesics} and was used there, as it will be here, in 
defining a ``scattering map''; see below.

\subsection {Monopoles and HE-Monopoles}   
\label{sec:HE-monopoles}
\begin{definition}\label{def:Dirac-type}
 Let $Y$ be a three-manifold, equipped with a metric, and
$p$ be a point of $Y$. Let $R$ denotes the geodesic distance to $p$. Let
$(t,x,y)$ be coordinates centred at $p$ for which the metric is in these
coordinates of the form $\id + O(R)$ as $R\to0$.  Let $\psi, \theta $ be, as
above, angular coordinates on the spheres $R=c$, so that $R,\psi,\theta $
provide standard spherical coordinates on a neighbourhood $B^3$ of $p$ defined
by the inequality $R<c$. We say that a solution to the HE-Bogomolny equations     $(E,\nabla, \phi)$ on
$Y\setminus\{p\}$ has a \emph{singularity of Dirac type, with weight
$\vec{k}=(k_1,\ldots,k_n)$ at $p$} if 
\begin{itemize}
\item the restriction of the bundle $E$ to $B^3\setminus\{p\}$ is a sum of line
bundles $L_{k_1}\oplus\cdots\oplus L_{k_n}$, and 
\item one can choose unitary trivialisations of $E$ over the  two open subsets
$\theta\neq 0$ and $\theta\neq\pi$ of $B^3$, with transition function
$\diag(e^{ik_1\psi},\ldots,e^{ik_n\psi})$, as above, in such a way that,
in both trivialisations,
\begin{align*}
\phi &= \frac {i}{2R}\diag(k_1,\ldots,k_n) + O(1),\\
\nabla (R \phi) &= O(1).
\end{align*}
\end{itemize}

A solution to the HE--Bogomolny equations that has singularities of Dirac type is called an \emph{HE-monopole}.
\end{definition}

The last condition tells us, via the HE--Bogomolny equations, that the
curvature $F$ is $O(R^{-2})$. 

As was first pointed out by Kronheimer
\cite{kronheimerMSc} for $Y =\R^3$, and expanded upon by 
\cite{pauly} for arbitrary three-manifolds, these conditions correspond to the
connection matrices and curvature being bounded at the origin (modulo gauge) 
once they are ``lifted" to $B^4$ via the Hopf map 
\begin{equation} \label{Hopfmap}
 \begin{aligned}
 B^4 &\rightarrow B^3\\ 
(w_1= u_1+ iu_2, w_2= u_3 +iu_4)&\mapsto (t, x+iy ) = (w_1\overline w_1 -
w_2\overline w_2,
2 w_1w_2)
 \end{aligned}
\end{equation} 
This lift is performed as follows:   one considers the trivial
bundle $\hat E = B^4\times \C^n$ equipped with the  $S^1$-action: 
\begin{equation}
(w_1, w_2, v_1,\ldots, v_n)\mapsto (e^{i\theta}w_1,e^{-i\theta}
w_2,e^{ik_1\theta} v_1, \ldots, e^{ik_n\theta}v_n)
\end{equation}
Then the $S^1$ quotient maps $\hat E$ to $E$, covering the map $\pi\colon B^4
\rightarrow B^3$. Pauly \cite{pauly} then shows that one can choose an
appropriate metric above so that if (1) one defines a connection matrix over $B^4$
in a lifted trivialisation over the complement of the origin as the lift $\pi^*\nabla - \xi\otimes \pi^*\phi$, where $\xi$ is a suitable one-form
on $B^4$ (in the Euclidean case, $\xi = 2(-u_2du_1+u_1du_2 - u_4du_3+u_3du_4)$); in the general case, the form is modified by a term of order 2 at the origin) and (2) one applies a gauge transformation taking one from $S^1$-invariant trivialisations to one defined at the origin,
one obtains:
\begin{proposition}\label{prop:Pauly}
The lifting process gives a correspondence between
\begin{itemize}
\item Solutions to the Bogomolny equations on $B^3\setminus\{0\}$ with a Dirac type singularity at the
origin, with weight $\vec{k}$.
\item Anti-self-dual (ASD) connections on $\hat E$, invariant under the circle
action, smooth away from the origin, and represented in a gauge in which the
action is given as above by a connection matrix which is in $L^2_3$ (and so is
continuous). 
\end{itemize}
\end{proposition}
 We note that the regularity one obtains above translates
into $O(R^{-1})$ bounds for the connection matrices and for the Higgs
field downstairs.

\begin{remark}For the Dirac monopole, this process in fact simply has a Dirac monopole correspond to a flat connection upstairs; we note that since a flat connection is anti-self-dual in any metric upstairs, we have (locally) a Dirac monopole in any metric downstairs.\end{remark}

We will return to the HE--Bogomolny equations for this lifting process later on. In the mean time, we note that in the case we are considering the metric on $B^3$ is not generic. Indeed with a coordinate $z= x+iy$ on $\Sigma$, the metric on $\R^3$ is given by
\begin{equation}
g = 
\alpha(z,\overline z) dzd\overline z +dt^2
\end{equation}
One can choose $z$  centred at the singularity, so that 
\begin{equation}\label{alpha}
\alpha(z,\overline z)  = 1 + z\overline z f(z,\overline z)     
\end{equation}
The Bogomolny equations for $S^1\times \Sigma$  are equivalent to time invariant anti-self-duality equations for $ S^1\times S^1\times \Sigma$, with the metric $g + ds^2$, where $s$ denotes the natural coordinate on the extra circle. This metric is K\"ahler with respect to the complex structure with complex coordinates $ s+it,z$. One can  think of the lifts from $B^3$ to $B^4$ in terms of a ``virtual'' lift from $S^1\times B^3$ to $B^4$, with the forms $dx, dy, dt$ lifting in the normal way and  $ds$ replaced by $\xi$. This process works quite well, in that the lift preserves bitype  for two-forms, the space of $(i,j)$ forms correspond  to $(i,j)$ forms (when $i+j = 2$). In particular, if one   uses the fact that $d\xi$ is of type $(1,1)$, then ``holomorphically integrable'' connections below (that is, satisfying the equation $[\nabla_x-i\nabla_y,\nabla_t-i\phi]=0$) correspond to holomorphically integrable connections upstairs.

\subsection{The scattering map}

\begin{definition}Let $(E,\nabla,\phi)$ be a solution
to the Bogomolny (or HE--Bogomolny) equation on the
product $I\times U$ of a interval $I$ with a possibly
open Riemann surface $U$.  Let $t_0<t_1$ be two values of $t$ in $I$. 
 The \emph{scattering map} $R_{t_0,t_1}\colon 
E_{(t_0,x,y)}\to E_{(t_1,x,y)}$ is defined  by taking for each 
$\sigma_0\in E_{(t_0,x,y)}$, the unique solution $\sigma$ of 
$(\nabla_t-i\phi)\sigma=0$ with $\sigma(t_0)=\sigma_0$; 
one then sets $R_{t_0,t_1}(\sigma_0):=\sigma(t_1)$.  
\end{definition}

If  $(E,\nabla, \phi)$ is an HE-monopole on $S^1\times \Sigma$, restricting the connection to the
surfaces $\Sigma_t$ defines a $\bar\partial$-operator $\nabla^{0,1}_\Sigma$ on
the surface, and so gives $E_t$ the structure of a holomorphic bundle. Now let
$t<t'$. The fact that $[\nabla^{0,1}_\Sigma, \nabla_t-i\phi]=0$ means that the scattering map defines a holomorphic isomorphism $R_{t,t'}$ from $(E_t)|_U$ to  $(E_{t'})|_U$, for $U$ open, as long as the set $[t,t']\times U$ does not contain one of the singular points $p_i$.
In particular, for $t_i<t<t'<t_{i+1}$, the holomorphic bundles on $\Sigma_t,\Sigma_{t'}$   are globally isomorphic. 

We
want also to pass through singularities. To understand what the parallel
transport does in that case, consider the  model Dirac monopole of charge $k$. The scattering  
map $R_{-1,1}$ from $t=-1$ to $t=1$ defined for 
$(x,y)\neq(0,0)$ takes a particularly pleasant form.
In the holomorphic trivialisations introduced in Section \ref{sec:Dirac} for which  
$(\nabla_t-i\phi)=\partial_t$, it is simply given by the transition function $z^k$, as shown by Equation (\ref{eqn:zk}). 

We now consider the asymptotics of the scattering map near a singularity for a general $\U(n)$ HE-monopole on
$S^1\times\Sigma$. We do so in a trivialisation that satisfies both 
\begin{equation}
(\nabla_t-i\phi)s=0,\quad  (\nabla_x+i\nabla_y)s= 0.\label{hol}
\end{equation}
 In particular, it is holomorphic along the surfaces $\Sigma_t$. As before, we
call such a trivialisation a \emph{holomorphic trivialisation}.

\begin{proposition}\label{prop:scattering}
In holomorphic trivialisations at $t=\pm 1$, the scattering map is of the form
\begin{equation}
h(z)\diag (z^{k_1},\ldots,z^{k_n}) g(z),
\end{equation}
with $h$ and $g$ holomorphic and invertible, and the coordinate $z$ chosen so that the singularity is at $z=0$.
\end{proposition}

\begin{proof}
We compare our monopole with a sum of $n$ Dirac monopoles with 
\begin{align*}
E_0 &= L_{k_1}\oplus \cdots\oplus L_{k_n}\\
\phi_0 &= \frac {i}{2R}\diag(k_1,\ldots,k_n)
\end{align*}

There is then a natural monopole structure on $E_0^*\otimes E$ with
connection $\nabla$ and Higgs fields $\phi$, and with weights
$(k_i-k_j)$. In the trivialisations used in Definition
\ref{def:Dirac-type}, 
there is a natural identification of $E$ and $E_0$; under this
identification, we choose the (holomorphic) section $S$ corresponding to the
identity map along $t=0$ (away from $z=0$), and extended outwards by
integrating $(\nabla_t-i\phi)S=0$.

As explained above, solutions to Equation (\ref{hol}) correspond to
holomorphic,
$S^1$-invariant sections upstairs on $B^4$. Our section $S$ lies in the weight
space of weight $0$, and so it extends holomorphically and invertibly (since one
can do the same for $S^{-1}$) to the origin. 

One   has the scattering map $\diag (z^{k_1},\ldots,z^{k_n})$
for $E_0$; applying $S$ gives the scattering map $S(1,z)\diag
(z^{k_1},\ldots,z^{k_n})S(-1,z)^{-1}$ for $E$. If one changes the
holomorphic trivialisations at $t=1,-1$, one has the general form given above.
 \qed\end{proof}

\section{Monopoles and stable pairs} 
\subsection{From monopoles to a bundle pair}\label{sec:montopair}
Let $(E,\nabla, \phi)$ be, as above, an HE-monopole. As we have
noted, restricting the connection to the
surfaces $\Sigma_t$ defines a $\bar\partial$-operator $\nabla^{0,1}_\Sigma$, 
and so gives $E_t$ the structure of a holomorphic bundle, denoted
$\EE_t$;
simultaneously solving
$(\nabla_t-i\phi)s=0$ and $\nabla_\Sigma^{0,1}s=0$ defines a holomorphic
isomorphism $R_{t,t'}$ from $\EE_t|_U$ to  $\EE_{t'}|_U$, for $U$ open, as long
as the set $[t,t']\times U$ does not contain any of the singular points $p_i$.
The discussions of the previous section can be summarised in:

\begin{proposition}[Kapustin--Witten \cite{geometricLanglands}]
The monopole restricted to the slice $\Sigma_t$ defines a
holomorphic $\Gl(n,\C)$ bundle $\EE_t$, away from the $p_i$.
\begin{itemize}
\item If there are no singular time $t_i$ between $t$ and $t'$,
the scattering map $R_{t,t'}\colon \EE_t\rightarrow \EE_{t'}$
is an isomorphism. 
\item If only one $t_i$  lies between $t$ and $ t' $, then
$c_1(\EE_{t'})- c_1(\EE_t) = \trk{\vec{k}_i}$, and 
$R_{t,t'}\colon \EE_t\rightarrow \EE_{t'}$ is a meromorphic bundle map
which is an isomorphism away from $z_i$, and there exist near
$z_i$ trivialisations of $\EE_{t}, \EE_{t'}$ such that $R_{t,t'}$  is given by
$\diag((z-z_i)^{k_{i,1}},\dots, (z-z_i)^{k_{i,n}})$ in these
trivialisations. (Here we abuse notation by letting $z_i$ denote both a point and its coordinate.)
\item More generally for all $t,t'$, by composition, the scattering maps $R_{t,t'}\colon \EE_t\rightarrow \EE_{t'}$ are meromorphic bundle maps which are isomorphisms away from the points $z_1,\ldots,z_N$. 
\end{itemize}
\end{proposition}

\begin{proof}
 The proof is straightforward and follows from Proposition \ref{prop:scattering}. 
\qed\end{proof}

In particular, integrating around the circle, we have:

\begin{definition} The \emph{monodromy} $\rho_t$ of $\EE_t$ is the map
$R_{t,t+T}$.  
\end{definition} 

The monodromy $\rho_t$ is a meromorphic endomorphism of $\EE_t$.  It has singularities near
$z_i$ of the form $F(z)\diag((z-z_i)^{k_{i,1}},\dots, (z-z_i)^{k_{i,N}}) G(z)$,
where $F(z), G(z)$ are
holomorphic and invertible.

\begin{definition}\label{def:bundlepair} A \emph{bundle pair} $(\EE,\rho)$ is the 
datum of a holomorphic bundle $\EE$ on $\Sigma$ and a meromorphic endomorphism $\rho\colon
\EE\rightarrow \EE$ such that $\rho$ is an isomorphism outside of a finite set
of points.
\end{definition}

Thus $(\EE_t,\rho_t)$ is a bundle pair. For a bundle pair $(\EE,\rho)$, let us
suppose that $\rho$ fails to be regular at $p$.
 If the rank of $\EE$ is $n$, Iwahori's theorem
(see \cite[Chap.\ 8]{Pressley-Segal} and \cite{Iwahori-Matsumoto})
tells us that
near the singular point $p$, choosing a trivialisation, and a coordinate $z$ centred at $p$, one can find
invertible holomorphic matrices $F(z), G(z)$ and integers  such that the map
$\rho$ in our trivialisation factors as
\begin{equation}
\rho = F(z) \diag (z^{\ell_1},\ldots,z^{\ell_n})G(z),
\end{equation}
with the set $\vec{\ell}=\{\ell_1,\ldots,\ell_n\}$ as invariants.    

In the case that concerns us here, we adapt the notation given in
the introduction and define
\begin{definition} A bundle pair $(\EE,\rho)$ has type
$\bbk=\bigl((\vec{k}_1,z_1),\ldots,(\vec{k}_N,z_N)\bigr)$ if its non-regular
points are  the points $z_1,\ldots ,z_N$, and the map $\rho$ is of
the form  $F(z)\diag (z^{k_{j1}},\ldots,z^{k_{jn}})G(z)$, with $F$ and
$G$ invertible,  near $z_j$, in coordinates $z$ centred at $z_j$.
\end{definition}

We define the bundle pair associated to a monopole 
$(E,\nabla,\phi)$ as \[\HH(E,\nabla,\phi):=(\EE_0,\rho_0).\] 
From what precedes, we see that for the monopoles we are
studying, $\HH(E,\nabla,\phi)$ has type $\bbk$.

\subsection{Constraints coming from the U(1) case}
Let $Y=(S^1\times \Sigma)\setminus \{p_1,\ldots,p_N\}$.  If $(E,A,\phi)$ 
is a $\U(n)$ singular monopole on $Y$ with a singularity of weight
$\vec{k}_j$ at $p_j$, then $(\det(E),\tr(A),\tr(\phi))$  
is a $\U(1)$ singular monopole on $Y$ with a singularity of weight $\trk{\vec{k}_j}$ at $p_j$.
We shall see that there are constraints on $\U(1)$-monopoles, in particular on the weights at the
singularities and  the location of the singular
points; these restrictions  propagate to arbitrary monopoles by taking traces.

We consider a $U(1)$ monopole, that is, a triple $(E,\nabla,\phi)$, considered modulo
gauge transformations: 

1) $\phi$ is a harmonic purely imaginary function  ($d*d\phi=0$), of the form $ik_j/2r+O(1)$ 
 near the points $p_j$, and otherwise regular, ($r$ is the geodesic distance to the singularity)

2) $\nabla = d+A$ is a $\U(1)$-connection satisfying $dA= *d\phi$.

Note that once the location  of the singular points and the charges $k_j$ are
fixed, the function $\phi$ is  unique up to a constant. 

Let $r$ be small, and let $S^2_j(r)$ denote the sphere of radius $r$ around
$p_j$, and let $Y_r$
denote the subset of points in $Y$ that are at least distance $r$ from the
singular set.  Thus $\partial Y_r=-\bigcup_j S^2_j(r)$. 
We have
\begin{align*}
 \sum_{j=1}^N k_j &=\sum_{j=1}^N
\frac{i}{2\pi}\int_{S^2_j(r)}F_A\\ 
&=-\frac{i}{2\pi}\int_{\partial Y_r}*d\phi\\ 
&=-\frac{i}{2\pi}\int_{Y_r}d*d\phi = 0.
\end{align*}
By extension, all singular $\U(n)$ monopoles on $Y$ must satisfy
\begin{equation}\label{eqn:condition0}\sum_{j=1}^N \trk{\vec{k}_j}=0.
\end{equation} 
This equation is also satisfied by HE-monopoles.

There are also constraints on the locations of the singularities. Since the monodromy $\rho_t$ is a meromorphic bundle map, its determinant is a meromorphic function, hence its divisor is principal and therefore the locations of the singularities are constrained in the direction along $\Sigma$.  They are, in fact, also constrained in the circle direction.

\begin{proposition}  
For a singular $\U(n)$-monopole, we must have 
\begin{equation}\label{eqn:condition1}
\sum_{j=1}^N \trk{\vec{k}_j}t_j =c_1(E|_{\{0\}\times\Sigma}) T.\end{equation}    
For a singular $\U(n)$ HE-monopole with constant $C$, 
we must have
\begin{equation} \label{eqn:condition1HE} 
\sum_{j=1}^N \trk{\vec{k}_j}t_j =
T\Bigl(c_1(E|_{\{0\}\times\Sigma})
+\frac{Cn}{2\pi}\mathrm{Vol}(\Sigma)\Bigr).
\end{equation}
\end{proposition}

\begin{proof}
Let's start with the $\U(1)$ case. Because we want $*d\phi=F_A$, the integral of
$\frac{i}{2\pi}*d\phi$ on any compact $2$-cycle must be an integer, indeed the
first Chern class of the restriction of the bundle to that $2$-cycle.
This condition imposes constraints on the location of the singular points.

Let $p_i=(t_i,z_i)$.  Suppose for notational simplicity that none of the $t_i$
are $0$ and that the $t_i$ are distinct, and let
\begin{equation}
 \label{eqn:defk0}k_0=\frac{i}{2\pi}\int_{\{0\}\times\Sigma}*d\phi.
\end{equation} 
Moving $t$ past the singular points
$t_1,\ldots,t_j$ in turn, the integral
$\frac{i}{2\pi}\int_{\{t\}\times\Sigma}*d\phi$ becomes $k_0+k_1+\cdots+k_j$.
On the other hand, one has $0=\int_{S^1}\partial_t \phi$ in the circle
orthogonal to the surface, away from the singular points. 
Let us suppose that the length of the circle is $T$, and set $t_{n+1}=T$, $t_0=0$.
Integrating over $S^1\times \Sigma$  (removing small cylinders around the
singular points and taking a limit), we find that
\begin{align*}0 &=\sum_{i=1}^{N+1}\bigl(\sum_{j=0}^{i-1}k_j\bigr)(t_i-t_{i-1})
		=\sum_{j=0}^N\sum_{i=j+1}^{N+1}k_j(t_i-t_{i-1})
		=\sum_{j=0}^Nk_j(T-t_j)
		= k_0T -\sum_{j=1}^N k_jt_j.\end{align*}
Equation (\ref{eqn:condition1}) follows easily. The general $U(n)$ case follows by taking traces.
\qed\end{proof}

As a corollary of the proposition, the geometric interpretation for the constant $C$ is 
\begin{equation}\label{eqn:geometricC}
C=\bigl(-\frac{2\pi}{\mathrm{Vol}(\Sigma)}\bigr)\Biggl(\frac{\displaystyle c_1(E|_{\{0\}
\times \Sigma} )-\frac{\sum_{j=1}^N\trk{\vec{k}_j}t_j}{T}}{n}\Biggr).
\end{equation}
Therefore $C$ is determined by the location of the  singularities.

We pause for a few remarks that underline the strong parallels that exist
between the monopole geometry and the complex geometry in the Abelian case.  
Let us first see how one can deform from a given solution, fixing the locus of
the singularities. 

For the monopoles, in the $\U(1)$ case, once one has $\phi$,
one has $dA$ and hence the connection $A$ up to a closed imaginary form. The equivalence relation on connections is modification by an imaginary exact form.
Thus the deformations of $A$ are parameterised locally by $2g+1$
real parameters, the dimension of $H^1(S^1\times\Sigma,\R)$. The free
parameters correspond to the integrals of $A$ around cycles; one can modify $A$
by adding to it lifts of imaginary harmonic forms on the Riemann surface, and
multiples of the form $i dt$. In addition, one can add an imaginary constant to
$\phi$; that is all the freedom one has, since the difference of any two $\phi$
with the same asymptotics is bounded, and so a constant, by elliptic regularity
and the maximum principle.
In short, there is modulo gauge transformations a $2g+2$ real parameter space of
pairs $(A,\phi)$ once one has fixed the singularities.

The monopole yields the complex geometric data of a line bundle on
$\Sigma_0$, given by the restriction of the connection to $\Sigma_0$, and so an
element of the Jacobian; it also gives the monodromy $\rho$, which is independent of the line bundle and
which 
once the divisor is fixed is determined up
to a non-zero complex constant. In short, one sees that   
the moduli space of complex data is a $\C^*$-bundle over the Jacobian, and so has the same dimension as our monopole moduli. 

The constraints on the location of the singularities are similar also: the
locations of the poles in the $t$-direction are determined in effect by the 
integral of $dt$ from some base point. In the direction of the Riemann surface
$\Sigma$, the divisor $\sum_j k_jz_j $ must be principal, and    Abel's
theorem requires that its image under the Abel map lie in the period lattice.
These constraints on periods can also be seen to intervene in the construction
of an Abelian monopole, in a way similar to the way they intervene in the
classical proof of the existence of a holomorphic function in Abel's theorem
(see \emph{e.g.}, \cite[p.232]{gharris}). Indeed, in both cases one builds a logarithmic derivative of the Higgs field or of the function; in both cases one first gets the right singularities, then adjusts so that one has integer periods. A complete proof for the Abelian monopoles would lead us into too
long a digression, and in any case the general proof valid also in the
non-Abelian case is the subject of this paper.

\subsection{The stability conditions}\label{sec:stability}
We now define an appropriate notion of stability for our holomorphic objects.
Set 
\begin{equation} \vec{t} = (t_1,t_2,\ldots,t_N,T),\quad
0<t_1\leq t_2\leq\cdots\leq t_N\leq T.\end{equation}

\begin{definition}\label{defdegree} The $\vec{t}$-degree $\delta_{\vec{t}}(\EE,\rho)$ of a 
bundle pair $(\EE,\rho)$ of singular type $\bbk$
is defined by 
\begin{equation}
\delta_{\vec{t}}(\EE,\rho)=c_1(\EE) - \frac{\sum_{j=1}^{N}
\trk{\vec{k}_j}t_{j}}T.\end{equation} The $\vec{t}$-slope of a bundle pair
$(\EE,\rho)$ of rank $n$ is the quotient
\[\mu_{\vec{t}}(\EE,\rho)=\delta_{\vec{t}}(\EE,\rho)/n.\]\end{definition}

\begin{definition} A   bundle pair $(\EE,\rho)$  is \emph{$\vec{t}$-stable
($\vec{t}$-semi-stable)} if
any proper non-trivial $\rho$-invariant subbundle $\EE'$ 
satisfies  $\mu_{\vec{t}}(\EE',\rho)  < (\leq)
\mu_{\vec{t}}(\EE,\rho)$. A bundle pair is $\vec{t}$-polystable if it
is the sum of stable bundle pairs of equal $\vec{t}$-slope.\end{definition}

\begin{remark}\label{independence} The notion of  $\vec{t}$-degree, and hence 
stability, is
invariant under shifting the origin in the circle, as
$\sum_{j}\trk{\vec{k}_j}= 0$; it is also
invariant as one moves through the singularities,
as going through the point $p_j$ changes $c_1(\EE)$ by $\trk{\vec{k}_j}$, but
also shifts $t_j$ by $-T$. Indeed, if $(\EE,\rho)$ is obtained from an HE-monopole $(E,\nabla,\phi)$, then rewriting the degree as 
\[\delta_{\vec{t}}(\EE,\rho)=\frac{\displaystyle\sum_{j=1}^{N}(t_{j+1}-t_j)\bigl(c_1(\EE)+\sum_{i\leq j} \trk{\vec{k}_i}\bigr)}T,\]
one sees that $\delta_{\vec{t}}(\EE,\rho)$ is the \emph{average degree} (average in $t$) of the restrictions to $\{t\}\times \Sigma$ of the bundle $E$.
\end{remark}

In this paper we show the equivalence between $\vec{t}$-polystable bundle pairs and HE-monopoles 
with Dirac singularities.  We proceed inductively on the rank. We now prove it in one direction,
showing that an HE-monopole yields a $\vec{t}$-polystable pair. Note that stability is automatic in 
the case of rank one.

\begin{definition}
 A HE-monopole $(E',\nabla',\phi')$ with constant $C'$ is a sub-HE-monopole 
of $(E,\nabla,\phi)$ (with constant $C$) if $E'$ is a subbundle of $E$
preserved by $\nabla$ and $\phi$ and if $\nabla$ and $\phi$ 
 restricted to $E'$ are $\nabla'$ and $\phi'$\end{definition}
In particular, $C=C'$.

\begin{proposition} A bundle pair $(\EE,\rho)$  on $\Sigma$ corresponding through $\HH$ 
to an HE-monopole $(E,\nabla,\phi)$ on $S^1\times \Sigma$ with constant $C$ is 
$\vec{t}$-polystable, and $\vec{t}$-stable if $(E,\nabla,\phi)$ is irreducible, that is it 
admits no sub-HE-monopole.
\end{proposition}

\begin{proof}  
Let $\VV$ be a $\rho$-invariant proper subbundle of $\EE$.  We can choose this
subbundle to be $\vec{t}$-stable.
By Proposition \ref{prop:surjective},  it corresponds to an HE-monopole
$(V,\nabla,\phi)$ with constant $D$, and singularities of weight $\vec{l}_j$ at
the point $p_j$.  We do not assume that $V$ is a submonopole of $E$, although
it is clearly a subbundle of $E$ by construction. The HE-monopole
$(V^*\otimes E,\nabla,\phi)$ 
has constant
$C-D$, and has a global $\rho$ invariant section $s$. 

Let $\nabla_\Sigma$ denote the restriction of the connection to the Riemann
surface, and $\Delta_\Sigma$ the associated Laplacian. Using the  identities
\begin{equation}\label{eqn:identities}
\begin{aligned}\nabla_\Sigma^{1,0}\nabla_\Sigma^{0,1} &=  (\Delta_\Sigma
+iF_{\Sigma})\omega\\  
(\nabla_t+i\phi)(\nabla_t-i\phi) &=\nabla_t^2 +\phi^2   -i \nabla_t\phi,
\end{aligned} 
\end{equation} 
 we find that for the section $s$ on $\Sigma$,
\begin{align*}
 (C-D)|s|^2_{L^2}&=\int_{S^1\times \Sigma}
\scp{s,i(\nabla_t\phi-F_{\Sigma})s}\dvol\\
&\leq
  \int_{S^1\times\Sigma}\bigl(\scp{s,i\nabla_t\phi
s}+|\nabla_ts|^2+|\phi s|^2\bigr)\dvol
+\int_{S^1\times\Sigma}\bigl(\scp{s,-iF_{\Sigma}s}
+|\nabla_\Sigma s|^2\bigr)\dvol\\
&=\int_{S^1\times\Sigma}
\scp{s,(i\nabla_t\phi-\phi^2-\nabla_t^2)
s+(-iF_{\Sigma}-\Delta_\Sigma)s}\dvol\\
&=-\int_{S^1\times\Sigma}
\scp{s,(\nabla_t+i\phi)(\nabla_t-i\phi)s+
\omega^{-1}\nabla_\Sigma^{1,0}\nabla_\Sigma^{0,1}s}
\dvol\\
&=0.
\end{align*}

The third step involves an integration by parts; one checks that this causes
no difficulties at the singularities. Using the geometric interpretation of the HE constants given by
Equation
(\ref{eqn:geometricC}), we obtain
$\mu_{\vec{t}}(\VV,\rho)\leq \mu_{\vec{t}}(\EE,\rho)$
with equality only if the section $s$ is covariant constant
and intertwines the Higgs fields for $E$ and $V$, hence if $V$ is a
sub-HE-monopole.  Since the orthogonal
complement of $V$ is then also a sub-HE-monopole of constant $C$, its image $\VV$
in $\EE$
is also preserved by $\rho$ and we can therefore by induction decompose $\EE$
in a sum of $\rho$-invariant $\vec{t}$-stable bundles of same slope and
$(\EE,\rho)$ is $\vec{t}$-polystable.
The proof is now complete.\qed\end{proof}

We note that in other situations of this type of correspondence (see, \emph{e.g.}, Lubke--Teleman
\cite{Lubke-Teleman}), the proof of this direction
does not assume the converse for bundles of lower rank,
but simply uses an integral of what looks like a second
fundamental form. The approach used here allows us to simplify  
dealing with the asymptotics at the singularity.

\section{Equivalence between stable pairs and monopoles}
In this section, we prove Theorem \ref{thm:main} given in the Introduction.

\begin{theorem}\label{thm:main2}
Suppose that the points $p_1,\ldots,p_N\in S^1\times \Sigma$ project to $N$
different points on $\Sigma$. The map
\begin{equation}\label{eqn:themap}
 \begin{aligned}
  \HH\colon \MM_{k_0}^{ir}(S^1\times\Sigma,p_1,\ldots,p_N,\vec{k}_1,\ldots,\vec{k}_N)
       &\to\MM_{s}(\Sigma,k_0,\bbk,\vec{t})\\
(E,\nabla,\phi)&\mapsto (\EE_0,\rho_0)
 \end{aligned}
\end{equation} 
between the moduli space of irreducible HE-monopoles and the moduli space of
$\vec{t}$-stable pairs
described by Section \ref{sec:montopair} is a bijection.
More generally, 
the reducible HE-monopoles  correspond     bijectively to 
 $\vec{t}$-polystable, but unstable, pairs.
\end{theorem}

The proof of surjectivity and injectivity are tackled separately by
Propositions 
\ref{prop:surjective} and \ref{prop:injective} below.

\begin{proposition}\label{prop:surjective}
 Given a $\vec{t}$-stable pair $(\EE,\rho)$ on $\Sigma$ of type
$\bbk=\bigl((\vec{k}_1,z_1),\ldots,(\vec{k}_N,z_N)\bigr)$ and the singular
time data
$0<t_1\leq\cdots\leq t_N\leq T$. There is a singular HE-monopole on
$S^1\times \Sigma$ of with Dirac-type singularities of weight $\vec{k}_j$ at
$p_j=(t_j,z_j)$ for which $\HH(E,\nabla,\phi)=(\EE,\rho)$. 
\end{proposition}

\begin{proof}
The steps we follow in showing this proposition are as follows:
\begin{itemize}
\item We use $\rho$ to extend $\EE$ to a bundle $E$ on
$Y=S^1\times \Sigma \setminus\{p_1,\ldots,p_N\}$, with the correct degrees on
the spheres around the points $p_j$, that is holomorphic on all the slices
$\Sigma_t$, and  lifts to a holomorphic bundle $\bar{E}$ on the
complex manifold $X=S^1\times Y$, subset of $\bar X=S^1\times S^1\times \Sigma$,
invariant under the action of $S^1$ on the first factor $S^1$ of $\bar X$.
\item We have on $\bar E$ a holomorphic structure; thus for any hermitian
metric on the bundle, there is a unique unitary connection (the Chern
connection) compatible with the holomorphic structure.  We choose a hermitian metric  on
$\bar{E}$ whose Chern connection
around the $j$th singularity is that of 
the sum of Dirac monopoles of weights $\vec{k}_j$. 
\item This metric serves as an initial metric for the heat flow  of
Simpson's paper
\cite{Simpson-Hodge-structures}.  We take the limit as time goes to infinity to
produce a Hermitian--Einstein connection on
$\bar{E}$, invariant under the action of $S^1$ on the first factor of $X$, and
so descending to $Y$. This process gives us the connection we want on $X$, and,
reducing to $Y$, our HE-monopole.
\item Simpson's theorem does not immediately give us the regularity we need at
the singular points. To see that the singularities are indeed of Dirac type; we
finish the proof by lifting locally on three-balls $B^3$ surrounding the
singularities using the Hopf map $B^4\rightarrow B^3$.
\end{itemize}

The first step is to extend $\EE$ to a bundle on $Y$.  We have supposed,
without 
loss of generality, that none of the $t_j$ is zero.  Consider the projection
map $\pi$ from $\tilde Y = \bigl((-T,T)\times \Sigma\bigr) \setminus \bigcup_j
\bigl((-T,t_j-T)\cup(t_j,T)\bigr)\times\{z_j\}$ to $\Sigma$; take the lift
$\pi^*\EE$ to $\tilde Y$. As $Y$ can be obtained from $\tilde Y$ by an
identification $(t, z)\rightarrow (t+T, z)$, 
we define a bundle $E$ on $Y$ by identifying $(t, z, v), t\in (-T,0), z\in
\Sigma,
v\in \EE|_z$ to  $(t +T, z, \rho(z)v)$.    Since the holomorphic
map $\rho$ decomposes, by hypothesis, into 
$h(z) \diag_l (z^{k_{jl}}) g(z)$, in coordinates $z$ centred at the
points $p_j$, the bundle one obtains is then indeed a sum 
$L_{k_{j1}}\oplus\cdots\oplus L_{k_{jn}}$ in a punctured neighborhood of $p_j$.
Lifting to $X$, as the clutching functions are holomorphic, the
result is a holomorphic bundle $\bar E$. Since it is holomorphic, it has a
$\overline\partial$ operator on it.

For the second step, we specify a hermitian metric $K$ on $\bar E$.
This metric specifies the Chern connection: one has
$\nabla= d+ ((\partial K )K^{-1})^T$ in
a holomorphic trivialisation. We   need a metric $K$ whose associated connection
is close to a solution of the Hermitian--Einstein equations, and has the
singular behaviour that we want:  we therefore choose Hermitian 
metrics corresponding to lifts of a sum of
Dirac monopoles in a neighbourhood of the singularities. As we also want our metric to be $S^1$-invariant on $X$, we
define it on $E$ over  $Y$.

We can choose in $\Sigma$ a small disk $D_0$, and disks $D_j$ surrounding the  
$z_j$, such that all these disks are mutually disjoint. We now choose over $\EE
= E_0$ a metric $k_0$:   
trivialise $E$ over the complement of $D_0$, and with respect to this
trivialisation,
choose the metric $k_0=1$ 
on this complement, then extend to a metric $k_0$ over $D_0$. The
curvature $F_0$ of the induced connection is then  concentrated on  $D_0$; its
trace represents the first Chern class of $\EE
= E_0$. Lift
the bundle $E_0$, and the metric  $k_0$, to $\tilde Y$. We note that as we vary
$t$, the curvature $F_0$ can no longer represent 
the Chern class of $E_t$ as $t$ moves through the various singular points, as
this Chern class changes. This problem is solved by glueing in Dirac monopoles
in balls around the singular points, in such a way that after moving through the
singular point $p_j$ some curvature is added into the disk $D_j$.

More explicitly, for $j=
1,\ldots,N$ let $C_j$ be disks properly included in $D_j$,  and let $\epsilon$
be such that $4\epsilon<\min(t_1, t_2-t_1,\ldots, t_N-t_{N-1}, T-t_N)$. We
cover $ Y$ by open sets 
\begin{align*}
U_0 &:= \bigl((-2\epsilon, t_N+2\epsilon)\times \Sigma\bigr)\setminus
\bigl(\cup_j (t_j-\epsilon, t_N+2\epsilon)\times C_j)\bigr),\\
U_{N+1} &:= (t_N+\epsilon, T-\epsilon)\times \Sigma,\\
U_{j-} &:= \bigl((t_j-2\epsilon, t_j+2\epsilon)\times D_j\bigr) \setminus
\bigl((t_j ,
t_j+2\epsilon)\times \{z_j\}\bigr),\quad j=1,\ldots,N,\\
U_{j+} &:=\bigl( (t_j-2\epsilon, t_N+2\epsilon))\times D_j\bigr)
\setminus 
\bigl((t_j
+2\epsilon, t_j )\times \{z_j\}\bigr),\quad j=1,\ldots,N.\end{align*}

In the trivialisation given above over the complement of $D_0$, factor $\rho$
near each point $z_j$  as $\rho = h_j(z) \diag_l (z^{k_{jl}})g_j(z)$
with $g_j, h_j$ invertible,
where the coordinate $z$ is centred at $z_j$ and
is chosen so that the metric osculates at $z_j$ the Euclidean  metric with orthonormal coordinates $t,\Re(z), \Im(z)$.
An equivalent construction of the bundle $E$ is given by specifying the
transition functions ($f_{\alpha,\beta}$ over $U_\alpha\cap U_\beta$):
\begin{gather*} f_{0, j-}=  g_j,\quad f_{j-, j+} = \diag_l
(z^{k_{jl}}),\quad f_{0, j+}= g_j \diag_l (z^{k_{jl}}),\quad f_{j+,N+1}
= h_j,\\
f_{0,N+1} = \rho \text{ over } (t_N +\epsilon, t_N +2\epsilon),\text{ and }
f_{0,N+1} = 1  \text{ over } (T-2\epsilon, T-\epsilon).\end{gather*}
Now note that the bundle and its transition functions are those for sums of 
Dirac monopoles on $U_{j-}$ and $U_{j+}$; we choose the hermitian metrics
$\diag_l((R-t)^{k_{jl}})$ on $U_{j-} $, and
$\diag_l((R+t)^{k_{jl}})$ on $U_{j+}$, which are compatible under the
change of basis. In parallel, we have the metric lifted from $\EE$ on $U_0,
U_{N+1}$. Choosing a partition of unity, we patch all these metrics together
over
$Y$, taking the non-trivial changes of trivialisations into account.

The metric $K$ we have obtained can be lifted 
to a metric $\bar K$ on $\bar E$.
  
\begin{lemma}\label{lemma:constructed}
 The  pair $(\bar E, \bar K)$ constructed above has the following properties
\begin{itemize}
\item $\bar E$ is invariant under the   action of 
$S^1$ on $X$; this action complexifies to an action of $\C^*$ over $S^1\times
S^1\times (\Sigma\setminus \{z_1,\ldots ,z_N\}) $, with the action of the real
element $T\in \R\subset \C^*$ corresponding to $\rho$
 \item $\bar K$ is invariant under the action of $S^1$ on $\bar E$,
 \item in the neighborhood of the inverse image of the singular point $p_j$,
the pair $(\bar E,\bar K)$
corresponds to  a sum of Euclidean Dirac monopole of weight
$\vec{k}_j$,
\item $(\bar E,\bar K)$ satisfies a bound $|\Lambda F_{\bar K}|\leq
c<\infty$.
\end{itemize}
\end{lemma}

The first three properties follow by construction.  For the fourth, we note that
$\Lambda F_{\bar K}^\perp$ would be  $0$ in a neighbourhood of the singular
circles if the metric on $X$ in this neighbourhood were Euclidean.  Since we took
coordinates osculating the Euclidean metric to second order, we still have the
bound. We note that we could have glued in the Dirac monopoles in the $S^1\times\Sigma$ metric produced in Section 
\ref{sec:HE-monopoles}
instead of the Euclidean ones, in which case $\Lambda F_{\bar K}^\perp$ would indeed be zero in a neighbourhood of the singularities; we have used the Euclidean ones for explicitness. 

We note that the curvature 
of the 
Chern 
connection is concentrated on $(S^1\times S^1\times D_0) \cup
(\cup_{j=1}^N (S^1\times (t_j-2\epsilon,T)\times D_j))$.

We take $K$ as the starting point for Simpson's heat flow
\begin{equation}
 \begin{aligned}
  H^{-1}\frac{d H}{du}&=-i\Lambda F_H^\perp,\\
H_0&=K.
 \end{aligned}
\end{equation} 
The asymptotic behaviour of this heat flow is governed by the following theorem.

\begin{theorem}[Simpson \citep{Thm 1, p.\
878, case
$\theta=0$}{Simpson-Hodge-structures}]\label{thm:Simpson}
 Let $(X ,\omega)$ satisfy certain conditions given below in Lemma
\ref{lemma:conditions}, and suppose $E$ is an $S^1$-invariant   
bundle on $X$ with $S^1$-invariant metric $K$ satisfying the assumption that
$\sup |\Lambda F_K| < c$. Suppose 
$E$ is stable, in the sense that it arises from a stable pair on $\Sigma$. Then
there is a $S^1$-invariant metric 
$H$ with $\det(H) = \det(K)$, $H$ and $K$ mutually bounded, $\bar\partial
(K^{-1}H)\in L^2$, and such that $\Lambda F_H^\perp = 0$. 
In addition, if $R$ is the distance to one of the singularities, $R\cdot
d(K^{-1}H)$ is bounded by a constant.
\end{theorem}
 
 (The last sentence does not appear as part of Simpson's statement of the
 theorem, but is given in a remark after his Lemma 6.4)

Since $H_\infty$ is
$S^1$-invariant, we can quotient out and consider it on $E$ (over $Y$).  The
equation $\Lambda F_{H_\infty}^\perp=0$ then becomes the HE-Bogomolny equation
(\ref{eqn:HEBogomolny}).  The metric $H_\infty$ is of course obtained as the
limit $\lim_{u\to \infty} H_u$ of the heat flow.

Simpson uses a notion of stability slightly different from ours.  His degree is
\[\deg(E,K)=i\int_X \tr(\Lambda F_K).\]

\begin{lemma}\label{lemma:degree}
The two notions of degree coincide: $\deg(E,K)=T\delta_{\vec{t}}(\EE,\rho)$. More
explicitly,
\[i\int_X \tr(\Lambda F_K) =  Tc_1(E_0) - \sum_{j=1}^{N}
\trk{\vec{k}_j}t_{j},\] 

Furthermore, there a correspondence between holomorphic subpairs $(\VV,
\rho|_\VV)$ of $(\EE, \rho)$ and $S^1$-invariant holomorphic subbundles $\bar V$
of $\bar E$, and for these also the notions of degree coincide.
\end{lemma}

\emph{Proof of Lemma \ref{lemma:degree}.} Indeed, for $\hat E$, the quantity 
$\tr(\Lambda F_K) $ is equal to the lift of $\tr( F_\Sigma-\nabla_t\phi)$.
Integrating $\tr(F_\Sigma)$ gives, as above, $Tc_1(E_0) - \sum_{j=1}^{N}
\trk{\vec{k}_j}t_{j}$; integrating  $\nabla_t\phi$, starting with the
$t$-direction,  gives $0$.

We note that from the definition of $\hat E$, $\rho$-invariant subbundles $\VV$
of $\EE$ naturally give $S^1$-invariant subbundles $\hat V$ of $\hat E$; on the
other hand, if a subbundle $\hat V$ of $\hat E$ is $S^1$-invariant, it is also
$\C^*$-invariant, and one can define a pair $(\VV, \rho)$ by setting $\VV =
\hat V|_{\{(0,0)\}\times \Sigma}$ and using the time $T$ action of $\C^*$ to
define $\rho$.
The proof of the equivalence of degrees for $\VV$ and $\hat V$ goes through as for
the bundle as a whole.
\qed{(Lemma \ref{lemma:degree})}

\bigskip

\begin{lemma}\label{lemma:conditions}
 The manifold
$X=S^1\times\bigl((S^1\times\Sigma)\setminus\{p_1,\ldots,p_N\}\bigr)$ satisfy
the three conditions necessary to Simpson's Theorem \ref{thm:Simpson}:
\begin{enumerate}
 \item $X$ is a K\"ahler manifold of finite volume;
\item there is on $X$ a non-negative exhaustion function whose Laplacian is
bounded;
\item 
there is an increasing function $a\colon [0,\infty)\to[0,\infty)$ with
$a(0)=0$ and $a(x)=x$ when $x>1$, such that if $f$ is a bounded positive
function on $X$ with $\Delta(f)\leq B$ then 
\[\sup_X|f|\leq C(B) a\bigl(\int_X|f|\bigr),\]
and furthermore, if $\Delta(f)\leq 0 $ then $\Delta(f)=0$.

\end{enumerate}
\end{lemma}

\emph{Proof of Lemma \ref{lemma:conditions}.}
The first condition is obviously satisfied by the construction of $X$.

To construct the  non-negative exhaustion function on
$X$ subject of the second condition, we first build a function $f$
on $Y$ whose Laplacian is bounded. 
The wanted exhaustion function is the pull back of $f$ to $X$ via the
projection on $Y$.  If the Laplacian of $f$ on $Y$ is bounded,
then  the Laplacian of the corresponding pull-backed $f$ on $X$ is also bounded.

Let $R_j$ be the geodesic distance in $S^1\times\Sigma$ to the singularity
$p_j$. In the Euclidean case, there is an obvious candidate for $f$: let $f$ be
$1/R_j$ close to $p_j$ and extend it smoothly to the
rest of $Y$.  Since $1/R$ is
harmonic in $\R^3$, the Laplacian of $f$ on $Y$ is obviously bounded.
For other Riemann surfaces, $\Delta(1/R)$ has a term
behaving like $1/R$ so we have to be careful
and find a bounded function whose Laplacian kills that extra $1/R$
factor.  In \cite[Prop 3.2.2]{paulythesis}, Pauly proves that there is
a harmonic function $f_j$ on a neighborhood of the singularity $p_j$ such that
$f_j=1/R_j+O(1)$.  This function is exactly what we are looking for, and
extending all the $f_j$ to $Y$ we find a function $f$ whose Laplacian is
bounded.
The obtained function $f$ is exactly the type of exhaustion function we are
looking for.

Simpson proves in \cite[Prop 2.2]{Simpson-Hodge-structures} that the third
condition is fulfilled for the smaller space $X^o=(T^2\times \Sigma)\setminus
\bigcup_j T^2\times\{z_j\}$.  Since $X^o$ is dense in $X$, the condition is also
fulfilled for $X$.
\qed(Lemma \ref{lemma:conditions})

\bigskip

In the case where $K$ is the metric we have carefully constructed above, the
limiting $H=H_\infty$ this theorem gives us yields a solution on $Y$ to the
HE-monopole equations.
We need to understand why the corresponding HE-monopole has the desired Dirac
monopole behavior at the poles. For this, as above, we can use the local construction
exploited both  by Kronheimer \cite{kronheimerMSc} and Pauly \cite{pauly} and
explained on page \pageref{Hopfmap}:
one considers the quadratic map $\pi\colon B^4 \rightarrow  B^3 $ given by 
\begin{equation}\pi(w_1,w_2)=( t =  |w_1|^2 - |w_2|^2, z= x+iy= 2w_1 w_2,)\end{equation}
This defines a lift of forms $\pi^*$; as noted above in Section 2.3, if we add an extra variable $s$ to $B^3$, expanding to $S^1\times B^3$, we can write the HE--Bogomolny equations for $\nabla,\phi$ as the HE equations for $\tilde\nabla = \nabla+\phi ds$ (recall that our complex coordinates are $t-is, z$); if we make the formal definition $\pi^*ds = \xi$, where $\xi$ as above is the $S^1$ invariant form
$$\xi = \frac{1}{i}(w_1d\overline w_1 -\overline w_1dw_1 -w_2d\overline w_2 +\overline w_2dw_2)$$ then the process used by Kronheimer and Pauly to smooth out the Dirac singularities associates to $(\nabla = d+A, \phi)$ the ``lift'' on $B^4$: $\widehat \nabla = \pi^*\tilde\nabla =  \pi^*\nabla+\pi^*\phi \xi$; the curvatures are related by 
\[F_{\widehat\nabla} = \pi^*F_{\widetilde\nabla} + \pi^*\phi d\xi.\] 
Now we have for an HE-monopole on $B^3$ an equation given by asking that, after lifting to $S^1\times B^3$, the projection of the curvature onto the space of self-dual two-forms, with kernel the space of anti-self-dual two-forms, take on a specified value. We can ask what this projection corresponds to on $B^4$. Unlike Pauly, we do not modify our form $\xi$ and the metric on $B^4$; we keep the standard Euclidean form and metric, and simply consider how the equation varies.

Keeping the formal lift $\pi^*(ds) =\xi$, we have the lifts (dropping the $\pi^*$):
\begin{align*}dz &= 2(w_1dw_2 + w_2dw_1), &dt-ids & = 2\overline w_1dw_1 -2 \overline w_2 dw_2,\\
d\overline z & = 2(\overline w_1d\overline w_2 + \overline w_2d\overline w_1), &dt+ids & = 2w_1d\overline w_1 -2  w_2 d\overline w_2.
\end{align*}
Therefore the bitypes above and below correspond; in particular the spaces of $(2,0),(0,2)$ and $(1,1)$ forms upstairs and downstairs correspond.

Let us now look at the $(1,1)$ forms under this formal pullback.
We have the lift of the K\"ahler form
\begin{align*}
\pi^*(\Omega)= \frac{i\alpha}{2} dz\wedge d\overline z -dt\wedge\xi 
  =2i\Bigl[& (\alpha |w_2|^2 + |w_1|^2)d w_1\wedge d\overline w_1 
            +(\alpha |w_1|^2 + |w_2|^2)d w_2\wedge d\overline w_2 \\
           &+(\alpha-1)\bigl(w_2\overline w_1d w_1\wedge d\overline w_2 
            +w_1\overline w_2 d w_2\wedge d\overline w_1\bigr)  \Bigr]
\end{align*}
and the lift of the three anti-self-dual forms
\begin{gather*}
\begin{aligned}
\bar\epsilon_1:= \frac14\Bigl(dz\wedge d\overline z -\alpha (dt-i\xi)\wedge (dt+i\xi)\Bigr) 
= & ( |w_2|^2-|w_1|^2)(d w_1\wedge d\overline w_1-d w_2\wedge d\overline w_2) \\
&+(1-\alpha) (|w_1|^2 d w_1\wedge d\overline w_1 + |w_2|^2 d w_2\wedge d\overline w_2 )\\
&+(1+\alpha) \bigr(w_2\overline w_1d w_1\wedge d\overline w_2 +w_1\overline w_2 d w_2\wedge d\overline w_1\bigr),
\end{aligned}\\
\begin{aligned}
 \bar\epsilon_2:=\frac14\Bigl(dz\wedge  (dt+i\xi)\Bigr) 
= &\phantom{-}w_1  w_2(d w_1\wedge d\overline w_1-d w_2\wedge d\overline w_2) 
 -w_2^2 d w_1\wedge d\overline w_2+ w_1^2 d w_2\wedge d\overline w_1, \\
\bar\epsilon_3:=\frac14\Bigl(
 d\overline z\wedge  (dt-i\xi) \Bigr)
= & -\overline w_1  \overline w_2(d w_1\wedge d\overline w_1-d w_2\wedge d\overline w_2) 
 -\overline w_1^2  d w_1\wedge d\overline w_2 +  \overline w_2^2 d w_2\wedge d\overline w_1.
\end{aligned}
\end{gather*}

Dividing by $4(|w_1|^2+|w_2|^2)$ 
the  lift of K\"ahler form, we have
\begin{align*}\tilde\Omega:=&\frac i2\bigl(d w_1\wedge d\overline w_1 + d w_2\wedge d\overline w_2\bigr)\\
&+\frac{ i(\alpha-1)}{2(|w_1|^2+|w_2|^2)} \bigl(|w_1|^2d w_2\wedge d\overline w_2  + |w_2|^2d w_1\wedge d\overline w_1  +w_2\overline w_1d w_1\wedge d\overline w_2 + w_1\overline w_2d w_2\wedge d\overline w_1\bigr).
\end{align*} 
Let $\omega$ be the standard K\"ahler  form  for the Euclidean metric upstairs, let $R^2=|w_1|^2+|w_2|^2$ and let $Q$ be the quadratic expression in $w_i,\bar{w}_i$ such that $\tilde{\Omega}=\omega+(\alpha-1)\frac{Q}{R^2}$.

Interestingly,  the basis $\{\tilde\Omega,\bar\epsilon_1,\bar\epsilon_2,\bar\epsilon_3\}$ of $\bigwedge^{1,1}$ is orthogonal for the usual Euclidean inner product on $B^4$, and for that inner product $|\tilde\Omega|^2=(\alpha^2+1)$.  Therefore the projection operator on the linear subspace spanned by $\tilde\Omega$ can be written 
\[P_{\tilde\Omega}(F)=\frac{\scp{F,\tilde\Omega}}{\scp{\tilde\Omega,\tilde\Omega}}\tilde\Omega.\]
One can check that in fact, 
\[P_{\tilde\Omega}(F)=\frac{\scp{F,\omega}}2\omega+\frac{(\alpha-1)}{(\alpha^2+1)}\left(\frac{\scp{F,Q}\omega+\scp{F,\omega}Q}{R^2}+(\alpha-1)\frac{\scp{F,Q}}{R^4}Q-\frac{(\alpha+1)}{2}\scp{F,\omega}\omega\right).\]

For $k\in \N$, let $p_k$ symbolically represents any homogeneous polynomial of degree $k$ in the $w_i,\bar{w}_i$ and let $p_\infty$ be any smooth function.  Therefore $p_k+p_k=p_k$ and $p_kp_l=p_{k+l}$.  Using this formalism, when $\alpha$ is normalised as in Equation  (\ref{alpha}), we have $\alpha=1+p_4p_\infty$.

For $k$ finite, notice that, if $\partial$ represents any derivative with respect to $w_i,\bar{w}_i$, we have
$\partial\bigl(\frac{p_k}{R^j}\bigr)=\frac{p_{k+1}}{R^{j+2}}$.
Since $\frac{p_k}{R^j}\in C^0$ if $k-j>0$, we have $\frac{p_k}{R^j}\in C^{k-j-1}$. As
\[P_{\tilde\Omega}(F)=p_\infty+p_\infty\frac{p_6}{R^2}+p_\infty\frac{p_{12}}{R^4},\]
 the coefficients of the projectors are in $C^3$.



Now the equation for HE connections below is 
\[P_\Omega(F_{\widetilde \nabla})= C\id_E\Omega\]
Lifting, the equation becomes
\[P_{\tilde\Omega}(F_{\widehat\nabla} - \pi^*\phi d\xi) = \pi^*(C\id_E\Omega).\]
We note that if $\alpha$ is uniformly one (the Euclidean case), $\Lambda( d\xi)=P_\omega(d\xi)= 0$; for $\alpha$ of the form $1+ w\overline w f(w,\overline w)$, $\Lambda( d\xi)=P_{\tilde\Omega}(d\xi)$ is of the form (bounded)(quartic) near the origin. 

The hermitian connection $\widetilde\nabla$ of interest to us is obtained from an initial hermitian connection
$\widetilde\nabla_0$, the Chern connection for a metric $H_0=K$, by keeping the same $(0,1)$ part and modifying the $(1,0)$ part so that $\widetilde\nabla$ is the Chern connection for a modified metric $H_\infty=H_0h$: this gives for the connection matrices (Simpson, lemma 3.1):
\[A^{0,1} = A_0^{0,1}, A^{1,0} = A_0^{1,0} + h^{-1}\nabla_0^{1,0}  h\]
In particular, for the Higgs field
\[\phi = \phi_0 +\frac{i}{2} h^{-1}\nabla_{0,t}h\]
For the  (1,1) component of the curvature, one has
\[ F_{\widetilde\nabla}^{1,1} = F_{\widetilde\nabla_0}^{1,1} + \nabla_0^{0,1}(h^{-1}\nabla_0^{1,0}  h)\]
Now lift this to $B^4$: one has the equation
\[ F_{\widehat \nabla} - \pi^*\phi d\xi = F_{\widehat\nabla_0} - \pi^*\phi_0 d\xi + 
\widehat\nabla_0^{0,1}(\pi^*h^{-1}\widehat\nabla_0^{1,0} \pi^* h)+ \pi^*h^{-1}[\pi^*\phi \overline \partial \xi, \pi^*h]\]

Since the original upstairs connection, a sum of the flat connections corresponding to Dirac monopoles,  has zero curvature in the neighbourhood of the origin, we have in this neighbourhood
\[ F_{\widehat \nabla} - \pi^*\phi d\xi = - \pi^*\phi_0 d\xi + 
\widehat\nabla_0^{0,1}(\pi^*h^{-1}\widehat\nabla_0^{1,0} \pi^* h)+ \pi^*h^{-1}[\pi^*\phi \overline \partial \xi, \pi^*h]\]
 We have the equation
\[\Lambda(- \pi^*\phi_0 d\xi + 
\widehat\nabla_0^{0,1}(\pi^*h^{-1}\widehat\nabla_0^{1,0} \pi^* h)+ \pi^*h^{-1}[\pi^*\phi_0 \overline \partial \xi, \pi^*h]) = \pi^*(C\id_E\Omega).\]
When $\alpha$ is uniformly one, the forms $d\xi , \overline \partial \xi$ are anti-self-dual, and the equations reduce to 
\begin{equation}\hat D(h) \equiv\Delta h -2i\Lambda \widehat\nabla_0^{0,1} h h^{-1} \widehat\nabla_0^{1,0}h=0.\end{equation}

When $\alpha$ is not uniformly one, we have an elliptic equation $\hat D(h) = 0$, a deformation of the one above, whose coefficients are $C^1$ (taking into account the poles of $\phi_0$ and the behaviour of $\Lambda(- \pi^*\phi_0 d\xi)$ ).

Now let us recall that $h$ is obtained as $h_\infty$ from a heat flow $h_u$ downstairs; $h$ is smooth away from the singularities. Upstairs, $h_u$ solves the heat equation $\partial_uh_u =\hat D(h_u) $.  Now take upstairs a four-ball around the singular point, mapping to a three-ball downstairs, and take as initial conditions for the heat flow  $\partial_u\hat h_u =\hat D(\hat h_u)$ the value $\hat h_0=h_0= 1$, and as boundary condition the Dirichlet condition $\hat h = h$. Applying the work of Donaldson \cite{Donaldson-boundary}, or simply again the results of Simpson, one obtains a $C^1$ solution $\hat h_u$, which is $S^1$-invariant as the initial and boundary conditions are so, and which satisfies the same boundary conditions and initial conditions as $h(t)$. One again has a limit $\hat h =\hat h_\infty$, solution to $\hat D(\hat h) = 0$

Both $h$ and $\hat h$ descend to the three-ball, and solve the HE--Bogomolny equations there. One then can refer to the lemma in Simpson \cite[p.~893]{Simpson-Hodge-structures}, that tells us that one has uniqueness if solutions are bounded, which they are. Thus $\hat h = h$, telling us that the global  $h$ given by Simpson's result has the required smoothness at the singular points (as $\hat h$ does) to ensure that the Higgs field and its covariant derivative have the correct Dirac type singularities.  
\qed\end{proof}

We now have a HE-monopole corresponding to our initial data; we now must check
that it is unique.

\begin{proposition}\label{prop:injective}
Given two monopoles $(E,\nabla,\phi)$ and $(E',\nabla',\phi')$ yielding
isomorphic holomorphic data.  Then the two monopoles are isomorphic.  Hence the
map $\HH$ given by Equation (\ref{eqn:themap}) is injective.
\end{proposition}

\begin{proof} Let $\HH(E,\nabla,\phi)=(\EE,\rho)$ and
$\HH(E',\nabla',\phi')=(\EE',\rho')$.  Since the holomorphic data are
isomorphic, the associated holomorphic bundles 
$\EE,\EE'$ are isomorphic by a holomorphic map $\tau$, 
in a way that intertwines $\rho$ and $\rho'$.  The same holds more generally 
for  $\EE_t,\EE_t'$. The map $\tau$ thus also aligns the 
corresponding eigenspaces of $\rho, \rho'$. One then has an isomorphism 
$\hat\tau$ from $E$ to $E'$ over $S^1\times\Sigma $.
One can combine the two monopoles and get a monopole 
$(E^*\otimes E',\hat\nabla =-\nabla\otimes \id 
+\id \otimes \nabla',\hat \phi= -\phi\otimes \id +\id \otimes \phi')$.

Consider $\hat\tau$ as a section of $E^*\otimes E'$.
We already know that $\hat\tau$ is in the kernel of  
$\hat \nabla_\Sigma^{0,1}$ and $\hat\nabla_t - i\hat\phi$. 
Using the identities (\ref{eqn:identities}), we find that
\begin{align*}
0 &=-\int_{S^1\times\Sigma}
\scp{\hat \tau,(\hat\nabla_t+i\hat\phi)(\hat\nabla_t-i\hat\phi)\hat \tau+
\omega^{-1}\nabla_\Sigma^{1,0}\nabla_\Sigma^{0,1}\hat \tau}
\dvol\\ 
&=\int_{S^1\times\Sigma}
\scp{\hat \tau,(-\hat\phi^2-\hat\nabla_t^2 - \hat\Delta_\Sigma)\hat \tau}\dvol\\
&=\int_{S^1\times\Sigma}
\scp{\hat\phi\tau,\hat\phi\tau}+ \scp{\hat\nabla_t\tau,\hat\nabla_t\tau}
+\scp{\hat\nabla_\Sigma\hat \tau,\hat\nabla_\Sigma\hat \tau} \dvol.
\end{align*}
Hence $\hat\tau$ is covariant constant.  As a map $E\to E'$, it intertwines the
two Higgs fields.  Hence the two monopoles are isomorphic.
\qed\end{proof}

\section{Moduli}
\subsection{HE-Monopoles on a three-torus}
 It turns out that the moduli space of
stable pairs has already been extensively studied for curves $\Sigma$ of genus
one in the context of integrable systems; see in particular \cite{HuMa2,HuMa1}.

\begin{definition} The pair $(\EE,\rho)$ is \emph{simple} if any section
 of $\End(\EE)$ commuting with $\rho$ is a multiple of the identity.
\end{definition}

\begin{proposition}  If $(E,\rho)$
is $\vec{t}$-stable, then it is simple.
\end{proposition}

The proof is the usual one: if it is not simple, then  the eigenbundles and
generalised eigenbundles  of a section $\sigma$ that commutes with $\rho$ but is not a multiple of the identity occur as both quotients and
subbundles of $\EE$, and are $\rho$-invariant. One has for any $\rho$-invariant subbundle $\FF$ and quotient $\QQ= \EE/\FF$, that $c_1(\EE) = c_1(\FF)+ c_1(\QQ)$. On the other hand, if $\rho_\FF$, $\rho_\QQ$ denote the endomorphisms on $\FF$, $\QQ$ induced by $\rho$,
the fact that $\det(\rho) = \det(\rho_\FF) \det(\rho_\QQ) $ tells us that $\trk{\vec{k}_j}(\EE) = \trk{\vec{k}_j}(\FF) + \trk{\vec{k}_j}(\QQ)$. One than has $\delta_{\vec{t}}(\EE,\rho)= \delta_{\vec{t}}(\FF,\rho_\FF) + \delta_{\vec{t}}(\QQ,\rho_\QQ)$, so that the existence of a non-trivial eigenbundle or
generalised eigenbundle  indeed contradicts stability.

Remark that for a stable pair with non-zero $\rho$ to exist, the
divisor $D= \sum_{j=1}^{N} \trk{\vec{k}_j}z_j$, of
degree zero by Equation (\ref{eqn:condition0}), must be a principal divisor
since it is the divisor of the determinant of $\rho$.

The paper \cite{HuMa2} describes the moduli of simple pairs $(\EE,\rho)$ for
arbitrary complex reductive groups. We adapt the result here for
$\Gl(n,\C)$. We begin by noting that there is a spectral curve $S^0=
S^0_{(\EE,\rho)}$ in $\Sigma\times \C$ defined by 
\begin{equation}
\det (\rho(z)-\lambda\id)=0.
\end{equation}
The curve $S^0$ extends to a closed curve $S$  in $\Sigma\times \bbp^1$; its
intersection with $\Sigma\times \{0,\infty\}$ occurs over the divisor
$\sum_jz_j$. Let $Char(\bbk)$ denote the family of curves obtained from our
simple pairs (recall the definition of $\bbk$ from page \pageref{def:bundlepair}); we do not describe it here with any thoroughness, 
referring instead to \cite{HuMa2}, except to note  
the fact that the orders $k_{jj'}$ of $\rho$ at the points $z_j$ constrain the
intersection of the spectral curve at $z_j$ with $\Sigma\times \{0,\infty\}$. 
Let $D_+=\sum_{k_{jj'}>0}k_{jj'}z_j$. 
The family $Char(\bbk)$ is the family of curves in the linear system
$|\pi_1^*({\OO}(D_+))\otimes \pi_2^*({\OO}(n))|$ satisfying these constraints.
 
For $\vec{k}=(k_1,\ldots,k_n)$ with $k_1\geq\cdots\geq k_n$,
set 
\begin{equation}\label{def:contribution}
[\vec{k}] = \sum_{a<b}k_{a}-k_{b} =
\sum_{a=1}^{n-1}(k_{a}-k_{a+1})a(n-a).
\end{equation}

\begin{theorem}[Main result of \cite{HuMa2}, adapted for $\Gl(n,\C)$]
\label{thm:elliptic}
Let $D$ be a principal divisor, and suppose that $\bbk\neq 0$. 
The moduli space ${\MM}_s(\Sigma, \bbk, k_0)$
of simple pairs $(\EE,\rho)$ of type $\bbk$, with $\EE$ of degree $k_0$ is
smooth, of complex dimension 
$ 2+\sum_{i=1}^N [\vec{k}_i] $. It has
a holomorphic symplectic structure, and the map 
\begin{equation}
{\MM}_s(\Sigma,\bbk,k_0)\to Char(\bbk)
\end{equation}
is Lagrangian, with generic fibre a smooth compact Abelian variety.
\end{theorem}

\subsection{Gauge theoretic dimensions}
Thus, in the case of $M=S^1\times T^2$,   Theorem \ref{thm:main} relates    
the moduli space of HE-monopoles to a complex moduli space described by Theorem
\ref{thm:elliptic}. In particular, we see that the moduli space for $\bbk\neq 0$ has real dimension
\begin{equation}
 \dim_\R{\MM}_{k_0}
(T^3,p_1,\ldots,p_N,\vec{k}_1,\ldots,\vec{k}_N)=4+2\sum_{j=1}^N[\vec{k}_j] .
\end{equation} 

The tangent space to the moduli space can also be understood from a gauge theoretic
point of view. In this context first order deformations of our HE-monopoles (modulo
gauge) correspond to the kernel of a complex
\begin{equation}\label{complex}
D^* +dB\colon \Omega^1(\ad(E))\oplus (\ad(E))\rightarrow \ad(E)\oplus
\Omega^2(\ad(E)
\end{equation}
where $D^*$ fixes the gauge infinitesimally, and $dB$ is the derivative of the
Bogomolny equation. In our case of simple bundles, the cokernel of this complex 
is of constant real dimension 4, while the index is given by a result of Pauly, 
for a quite general three-manifold.

\begin{theorem}[Extension of Pauly's result in \cite{pauly} to
higher order groups]\label{thm:vdim}
Let $M$ be a compact, oriented, and connected Riemannian $3$-manifold.  Fix
$N$ points
$p_1,\ldots,p_N\in M$, and $N$ sequences $\vec{k}_1,\ldots,\vec{k}_N$
of $n$ integers.  Then the real index of the deformation complex
(\ref{complex}) of a HE-monopole $(E,\nabla,\phi)$ is 
$2\sum_{j=1}^N[\vec{k}_j]$.
\end{theorem}

The proof of Pauly, originally written for $\SU(2)$, extends with little work to
this more general case.  By exhibiting parametrices, Pauly shows that
despite the presence of singularities, the complex is Fredholm. This
result relies on the fact that near the singularities, the asymptotic behaviour
guarantees that a local lift  from the three-ball to the four-ball using the
Hopf fibration is non-singular. 
The index can then be obtained by an excision argument from
the case with no singularities.  In that (compact) case, since the
dimension is odd, Atiyah--Singer's index theorem tells us the index is
$0$.  The contribution to the index given by the singularities 
translates the problem into one over the three-sphere, and Pauly then uses the lift
to the four-sphere to transform the calculation into that of an $S^1$
equivariant index.

We note also that, in the case of the three-torus, when there are no singularities ($\bbk = 0$), 
we are in essence reduced to the flat (Abelian) case: the index is zero, 
and both kernel and cokernel are of constant real rank $4n$. This result is confirmed 
by a parameter count: $3n$ parameters for a flat connection, which is a representation  
of $\Z^3= \pi_1(T^3)$ into (the maximal torus of) $U(n)$, and $n$ parameters for a
 constant Higgs field.

\subsection{Higher genus} Returning to our complex descriptions of the moduli, 
many of the techniques used for studying the moduli space in \cite{HuMa2} also
apply to the case of HE-monopoles over $S^1\times \Sigma$ with $\Sigma$ of higher genus, 
apart from the derivation of the Poisson structure and the existence
of an integrable system. 

We first note that our moduli space of stable pairs can be examined  as a subspace 
of a space
that has already been constructed. Consider the divisor $D_{max}=-\sum_j
\min_l(k_{jl})z_j$. Then $\rho$ is a section of $\End(\EE)(D_{max})$, and one
can realise the space of our $(\EE,\rho)$ (for $t=(t_i)$ small) 
as a subvariety of the space of stable pairs consisting of a bundle  and an endomorphism with
poles at $D_{max}$. This moduli space ${\NN}(k_0, D_{max})$ has been
constructed in \cite{Nitsure, Simpson-1et2}. We can thus study 
${\MM}(\Sigma,k_0,\bbk,\vec{t})$ as a subvariety of  ${\NN}(k_0, D_{max})$.

A first step, however, is to check that at least some of our spaces 
${\MM}(\Sigma,k_0,\bbk, \vec{t})$ are non-empty, provided again that the divisor
$D$ is principal, which it must be as it 
represents $\det(\rho)$. Suppose 
that $D$ is non-zero.  For each point $z_i$, choose a permutation
$\sigma_i$ of $\{1,\ldots,n\}$ and set $D_j= \sum_i k_{i,\sigma_i(j)}z_i$. Let
$L$
be a
line bundle on the curve, and set $L_j = L(D_1+\cdots+D_j), j=1,\ldots,n$. 
There is a
natural meromorphic map $\rho_{j+1}\colon
L_j\rightarrow L_{j+1}$ with divisor
$D_{j+1}$. In the same vein, the natural map $L\to L_1$ yields a map
$\rho_1\colon L_n\rightarrow L_1$ when premultiplied by $\det(\rho)^{-1}$. 
This map has 
divisor 
$D_1$. Let $E$ be the bundle $\oplus_jL_j$; 
it has degree $n \deg(L) + \sum_i(n-i+1)\deg(D_i)$. 
Let $s$ be the residue modulo $n$ of this degree; as usual,
it is just the value modulo $n$ that is of any importance. 

\begin{proposition} Suppose that the divisor $D$ is principal  and that the
permutations are such that not all the degrees of $L_j$ are the same. The moduli
space ${\MM}(\Sigma,mn+s,\bbk, \vec{t})$ is non-empty.
Suppose in addition that one of the $L_j$ has degree $n$ greater than any of the
others. Then one can also produce  an element of ${\MM}(\Sigma,mn+r,\bbk,
\vec{t})$ for any $r$.
\end{proposition}

\begin{proof} Consider the  bundle $E$ constructed above of degree $mn+s$.  Define the map
$\rho$ by 
\begin{equation}\label{cyclic}
\rho = \begin{pmatrix}
0& 0&0&\cdots&0&\rho_1\\
\rho_2& 0&0&\cdots&0&0\\
0&\rho_3& 0&\cdots&0&0\\
0&0&\rho_4&\ddots&0&0\\
\ddots&\ddots&\ddots&\ddots&\ddots&\vdots\\
0& \ddots&0&0&\rho_{n-1}&0\\
\end{pmatrix}
\end{equation}
It has the right polar structure. We can ask when the pair $(\EE,\rho)$ is
stable. The determinant of $\rho$ has divisor $D$;   away from the singular set of
$\rho$,  the eigenvalues of $\rho$, which are the $n$-th roots of $\det(\rho)$,
are distinct. Now consider a subbundle $V$ of rank $k$ that is
$\rho$-invariant; locally, it is a sum of eigenbundles. The projection of $V$
onto any sum $L_{i_1}\oplus L_{i_2}\oplus\cdots\oplus L_{i_k}$ has generically
non-zero determinant since it is  a Vandermonde-type determinant of
the eigenvalues involved.
Globally, $V$ is then a subsheaf of
$L_{i_1}\oplus L_{i_2}\oplus\cdots\oplus L_{i_k}$; this inclusion bounds the degree of $V$. If we
have chosen the $L_{i_j}$ of smallest degree, $V$ cannot
destabilise. 

To consider the case of general degree, one can take Hecke transforms $\EE'$ of
$\EE$ along the eigenspaces of $\rho$, obtaining subsheaves of $\EE$ and
reducing
the degree by $1,2,\ldots,n-1$ in turn. The result is still stable, as any
$\rho$-invariant subbundle $V'$ of $E'$, thought of as a subsheaf of $E$, then
satisfies $\deg(V')/\rank(V')< (-n +\deg(E))/n\leq \deg(E')/n$.\qed\end{proof}

One would like to get an idea of the Zariski tangent space of 
${\MM}(\Sigma,j,\bbk,\vec{t})$,
at the points of the moduli space we have just constructed. As in
\cite[Sec.\ 4]{HuMa2}, let us consider the subbundle 
\begin{equation} 
\Delta_\rho \ := \ \{(a,b) \mid \ a+Ad_\rho(b)=0\}
\end{equation}
of $\End(\EE)\oplus \End(\EE)$. Denote the quotient by
\begin{equation} 
\ad(\EE,\rho) \ := \ [ \End(\EE)\oplus \End(\EE)]/\Delta_\rho.
\end{equation}
We get the short exact sequence
\begin{equation} 
0\rightarrow \Delta_\rho \rightarrow [ \End(\EE)\oplus \End(\EE)] \rightarrow 
\ad(\EE,\rho) \rightarrow 0.
\end{equation}

The degree of $\ad(E,\rho)$ is computed by \cite[Lemma 4.9]{HuMa2} to be
$\sum_{i=1}^N [\vec{k}_i]$. As in \cite[Cor.\ 4.3]{HuMa2}, we have the following
proposition;
see also \cite{Biswas-Ramanan,Bottacin-symplectic-stable-pairs,Markman}.

\begin{proposition}
The infinitesimal deformations of the pair 
$(\EE,\rho)$ in ${\MM}(\Sigma,j,\bbk,\vec{t})$ 
are naturally identified by the first hyper-cohomology
of the complex (in degrees $0$ and $1$)
\begin{equation}
\End(\EE) \stackrel{\ad_\rho}{\longrightarrow}  \ad(\EE,\rho). 
\label{deformation}
\end{equation}
\end{proposition}

Roughly, the first hypercohomology combines both $H^1(\Sigma, \End(\EE))$, the
deformations of the bundle, and $H^0(\Sigma, \ad(\EE,\rho))$, the
deformations of $\rho$.

The dimension  $\bbh^1$ of the first hypercohomology is given by 
\[\bbh^1 = \bbh^0 + \bbh^2 -\chi(\End(\EE))+\chi(\ad(\EE,\rho)) =  \bbh^0 +
\bbh^2 +\sum_{i=1}^N [\vec{k}_i]\]
If $(\EE,\rho)$
is simple, the dimension $\bbh^0$ is $1$, as in \cite[Sec.\ 4]{HuMa2}.
The dual space to the second cohomology is given, again as in
\cite[Sec.\ 4]{HuMa2}, by the kernel of 
$\ad_\rho\colon\Delta_\rho\otimes K \to 
\End(\EE) \otimes
K$, that is the sections of $\End(\EE)\otimes K$ that commute with $\rho$. 
If $\rho$ is generically regular, this kernel is generated away from the poles by the powers $1, \rho,\ldots,\rho^{n-1}$ of $\rho$, tensored with $K$. Globally, the sections of the kernel are given as expressions $\sum_{i=0}^{n-1}a_i\rho^i$, where $a_i$ is a form such that the product $a_i\rho^i$ is holomorphic. 
In our case, the coefficient $a_i$ must have divisor greater or equal to
$\max(-D_1- D_2-\cdots-D_i, -D_2-D_3-\cdots-D_{i+1},\ldots,
-D_n-D_1-\cdots-D_{i-1})$. For fairly general choices of divisors, if one has
more than $g$ points, this condition forces $a_i=0$ for $i>0$ and $a_0$ to be
 a holomorphic one-form; one then has $dim(\bbh^2) = g$. This computation yields
\[dim(\bbh^1) = (g+1) +\sum_{i=1}^N [\vec{k}_i].\]
When the elements of  $\bbh^2$ live in the trace component $a_0$, as it is the
case here under our genericity assumptions on the divisors, it is shown in
\cite[Thm 4.13]{HuMa2} that the deformations are unobstructed and
the space
is smooth.

\begin{remark} If  $\EE$ has trivial determinant and $\det(\rho)=1$, one can
consider the space of deformations that preserve this property; this 
situation corresponds to $\SU(n)$ monopoles. (We note that this constrains the location of the singularities in the circle direction.) As the elements of $\bbh^0$,
$\bbh^2$ lie in the trace component, the deformations keeping one in $\Sl(n)$
lie in a $\sum_{i=1}^N [\vec{k}_i]$-dimensional space; the extra $g+1$
parameters for $\Gl(n)$ correspond to rescaling $\rho$ (one parameter) and tensoring $E$ by a
line bundle ($g$ parameters). \end{remark}

\subsection{\texorpdfstring{Higher genus: $\SU(2)$, $\U(2)$}{Higher genus: SU(2), U(2)}}
The preceding results are  only  rather partial: as one can see, the
combinatorics of the degrees is fairly complicated. We now consider the case of
$\Sl(2)$-bundles, with endomorphisms of determinant one; these correspond to 
$\SU(2)$-monopoles.  Let us suppose given a pair $(E,\rho)$, with singularites
of type $(k_i,-k_i), k_i>0 $ at points $z_i$. Let $D_+$ be the divisor
$\sum_ik_iz_i$. By what is now a fairly standard construction, one has a
spectral curve $\{(z,\lambda)\in \Sigma\times\PP^1\mid \det(\rho(z)-\lambda\id)=0\}$ in $\Sigma\times \PP^1$, giving a two
sheeted cover of $\Sigma$. In addition, one can define what is generically a
line bundle $L$ over the spectral curve; away from the poles of $\rho$, it is
the quotient sheaf, defined over the total space of the bundle ${\OO}(D_+)$ over
$\Sigma$ as
\[0\rightarrow E\buildrel{\rho-\lambda\id}\over{\rightarrow} E\rightarrow L.\]
We note that as the spectrum is invariant  under $\lambda\mapsto \lambda^{-1}$, the spectral
curve is the pullback of the graph of the function $tr(\rho)$ under the map
$f=\lambda+\lambda^{-1}$. One sees that the singular points of the endomorphism $\rho$ are
located  at the poles of $f$, and that the  branch locus $B$ for projection
of the spectral curve to $\Sigma$ is given by the inverse image of $\{-2, 2\}$. The
curve is smooth if all the points of $B$ occur with multiplicity one. The curve 
has an involution $i$, and the determinant form on the bundle $E$ identifies
$L^*$ with $i^*(L)(-B)$ (as in \cite{HitchinSDRiemann}),
so that $L$ lies in a suitable Prym variety.

Conversely, given the spectral curve $S$ and the line bundle $L$ lying in the Prym variety, one can reconstruct $E$ as a push-down $\pi_*(L)$, and $\rho$ as the endomorphism on $E$ induced by multiplication by $\lambda$ on $L$. If $S$ is smooth, stability is automatic: any invariant subbundle corresponds to a subset of the eigenvalues away from the branch points, and going around these branch points permutes them, so that an invariant subbundle is necessarily the whole bundle.

Thus
\begin{proposition}
An $\SU(2)$-monopole on $S^1\times\Sigma$, with singularities $\bbk$, yields 
a pair $(f, L)$ where $f$ is a meromorphic function with polar divisor   $D_+$, 
and $L$ is a sheaf over the double cover $S$ of $\Sigma$ branched over $B$, the 
locus defined by $f^{-1}(\{2,-2\})$. If all the points in the support of $B$ occur with 
multiplicity one in $B$, the spectral curve $S$ is smooth, and  $L$ is a line bundle  
on $S$ belonging to the Prym variety of sheaves on $S$ satisfying 
$L^*\simeq i^*(L)(-B)$, where $i$ is the natural involution.

Conversely, given $(f,L)$, where $f$ has polar divisor $D$, if all the points in the 
support of $B$ occur with multiplicity one in $B$, then the double cover $S$ is smooth. 
If $L$ belongs to the Prym variety, then the pair $(f,L)$ corresponds to a singular $\SU(2)$-monopole.
\end{proposition}

One can then count parameters. For $d_+> 2g-2$, the space of functions $f$ with 
divisor exactly $D_+$ has dimension
$d_++1-g$, for $d_+ = \deg(D_+)$ . If the locus $f^{-1}(\{2,-2\})$ consists of distinct
points, the genus of $S$, from the Riemann--Hurwitz formula, is then $2g-1+d_+$,
and so the Prym variety has dimension $g-1+d_+$; the $\SU(2)$ moduli space thus
has dimension $2d_+$ in all. 

We now can consider the more general case of a $\Gl(2,\C)$ bundle, and a map
$\rho$ with singularities $\bbk=\bigl(\bigl((k_{1,+},k_{1,-}),z_i\bigr),\ldots,\bigl((k_{N,+},k_{N,-}),z_i\bigr)\bigr)$, with $k_{i,+}\geq
k_{i,-}$; this pair correspond to a $\U(2)$ (HE)-monopole. There is again the
constraint imposes  by the requirement that the divisor $\sum_i(k_{i,+} + k_{i,-})z_i$
be principal, as it must be the divisor of $\det(\rho)$. The same constructions
give one a double cover $S$ of $\Sigma$, defined by the equation
\[\lambda^2 - tr(\rho)\lambda + \det(\rho) = 0.\]
As above, one has a sheaf $L$ over $S$, which is a line bundle if $S$ is smooth. $L$ no longer necessarily satisfies the Prym condition.

Again, one can count parameters. The function $tr(\rho)$ must have a divisor which is greater than $\sum_ik_{i,-}z_i$; if $d=-\sum_ik_{i,-}= \sum_ik_{i,+}$, then for $d>2g-2$ this gives $d+1-g$ parameters. The determinant is fixed, up to scale; this gives one extra parameter. On the line bundles, one has $2g-1+d$ parameters, giving $2d+g+1$ parameters in all: roughly, the parameters for the $\Sl(2,\C)$ case, plus a line bundle on $\Sigma$, plus a scale parameter for $\rho$.

\section{Monopoles on the product of a Riemann surface and an interval}
One can use the result to consider the case of a monopole on the product
$I\times \Sigma$ of a Riemann surface and an interval $I= [0,c]$. In this case,
the relevant holomorphic data is a pair of holomorphic bundles $E= E(0), E'=
E(c)$ and a meromorphic automorphism $\rho\colon E\rightarrow E'$. If $z_i, i=
1,\ldots,N$ are the singularities of $\rho$, choose $t_i$ so that $p_i= (t_i,
z_i)$ are the locations of our eventual monopole's singularities.

\begin{proposition} Given $\rho\colon \EE\rightarrow \EE'$, one can find a
complementary $\sigma\colon \EE'\rightarrow \EE$ with singularities at
$z_{N+1},\ldots,z_{N+M}$ such that $\rho, \sigma$ have disjoint singularities,
and $t_{N+1},\ldots,t_{N+M}$ such that $(\EE,\sigma\circ\rho)$ is a
$\vec{t}$-stable pair. One can also arrange for the HE-constant $C$ to be zero.
\end{proposition}

The proof is fairly simple. If one allows sufficiently many poles, one can find
a large number of $\sigma$; the trick is to ensure that the result is
$\vec{t}$-stable.
The simplest way to ensure stability is to arrange for there to be no  
$\sigma\circ\rho$-invariant subbundles at all, a more restrictive condition. As
pointed out above, invariant
subbundles are sums of generalised eigenspaces; if the spectral curve has a
branch point that permutes all of these eigenspaces with no invariant subset,
then we are
done. Now, one can prescribe the behaviour of $\sigma\circ\rho$ on any formal
neighbourhood of a given point,  providing one allows sufficiently
many poles
elsewhere;  one then asks that this point be a branch point that permutes all
the eigenspaces, which is ensured by prescribing the behaviour on a formal
neighbourhood of sufficiently high order by letting $\sigma$ be the composition
of a well chosen Jordan
form by $\rho^{-1}$.

\begin{corollary}
There is a monopole with singularities at $p_i,i=1,\ldots,N$ corresponding to
the holomorphic data  $\EE= E(0), \EE'= E(c), \rho\colon \EE\rightarrow \EE'$.
\end{corollary}

One simply embeds the interval $I$ in a circle $C$, and places the extra poles
(those of $\sigma$) in the complement of $I$. One then applies the theorem for
the circle, obtains a monopole on $C\times \Sigma$, and restricts to $I\times
\Sigma$.

The monopoles associated to a given set of holomorphic data  are not unique.
Indeed, one can complete by different maps $\sigma, \sigma'$, and there is no
reason why the monopoles one obtains by restricting to $I\times\Sigma$
should be the same.

\section{HE-Monopoles on a flat circle bundle}
If one has a flat principal circle bundle $Y$ over a Riemann surface $\Sigma$,
then any metric on $\Sigma$ and any choice of length $T$ of the circle yield a
canonical metric on the circle bundle: at a point $p$ on the circle bundle, one
takes a local flat lift of $\Sigma$ given by the connection, and the tangent
space of the circle bundle decomposes into a sum (that we stipulate to be
orthogonal) of the orbit direction and the tangent to the section. On the first
one has a natural metric inherited from the invariant metric we have chosen on
the circle, and on the latter the lift of the metric on $\Sigma$; the two
combine to give a metric on $Y$.

Now pass to the universal covering $\widetilde \Sigma\rightarrow \Sigma$, and lift our circle bundle $Y$ with flat connection to  $\widetilde Y\rightarrow \widetilde \Sigma$. One can choose a global flat section $S$ on this universal cover, avoiding the lifts of the singularities. This section trivialises the circle bundle; let $S$ correspond to $t=0$. Let $\Deck$ be the group of deck transformation of this covering.  For each $D\in\Deck$ , one has an element $x(D)$ of the circle such that the original flat circle bundle over $\Sigma$ is obtained by identifying $(z, t)$ and $(D(z), x(D)\cdot t)$. 

The restriction of the lift of an HE-monopole to $S$ yields a holomorphic vector bundle $\widetilde\EE$ over $\widetilde \Sigma$. This bundle descends to $\Sigma$; one uses $\nabla_t-i\phi$ to define a parallel transport $T(x(D))$ from the fiber at $(D(z), x(D)\cdot t)$ to $(D(z), t)$. The holomorphic bundle $\EE$ on $\Sigma$ is obtained by composing the natural identification of $\widetilde\EE_{(z, t)}$ and $\widetilde\EE_{(D(z), x(D)\cdot t)}$ with the transport from $\widetilde\EE_{(D(z), x(D)\cdot t)}$ to $\widetilde\EE_{(D(z), t)}$. These identifications intertwine the parallel transport $\widetilde\rho$ mapping $\widetilde\EE$ to itself, and so $\widetilde\rho$ descends to $\rho\colon\EE\rightarrow\EE$.

Let $\tilde z_i$ be a lift of the point $z_i$ to $S$ and let $t(\tilde z_i)$ denote the time one must flow from
the section $S$ to the lift  of the point $p_i$. If $U$ is the union of $\ell$ disjoint fundamental domains in $S$ for the covering map to $\Sigma$, we define a  degree by:

\begin{definition} The $(\vec{t},U)$-degree $\delta_{\vec{t}, U}(\EE,\rho)$ of a 
bundle pair $(\EE,\rho)$ of singular type $\bbk$
is defined by 
\begin{equation}
 \delta_{\vec{t},U}(\EE,\rho)= c_1(\EE) -
\frac{\sum_{\tilde z_j\in U}
\trk{\vec{k}_j}t(\tilde z_j)}{T\ell}.
\end{equation} 
\end{definition}
Let $\widetilde M_t\colon\widetilde Y\rightarrow \widetilde Y$ denote the action of $\exp(2\pi i t/T)$ on $\widetilde  Y$.
Suppose now that our circle bundle is such that its lift to a finite cover $\Sigma'$ of $\Sigma$ is trivial. Choosing $U$ to be the union of fundamental domains for $\Sigma$ corresponding to one fundamental domain of $\Sigma'$, one has an equivalent to Remark \ref{independence}: 

\begin{lemma}
If one translates $S$ to $\widetilde M_t(S)$, the degree does not change. \end{lemma}

\begin{theorem} Let $Y$ be a circle bundle that become trivial when lifted to a finite cover $\Sigma'$ of $\Sigma$.  
The moduli space ${\MM}^{ir}_{k_0}(Y,
p_1,\ldots,p_N,\vec{k}_1,\ldots,\vec{k}_N)$ of irreducible $\U(n)$ HE-monopoles on
$Y$ with $\EE$ of degree
$k_0$ and singularities at $p_j$ of type $\vec{k}_j$ maps bijectively to the
space ${\MM}(\Sigma,k_0, \bbk,\vec{t})$ of
$(\vec{t},U)$-stable holomorphic pairs $(\EE,\rho)$ with 
\begin{itemize}
\item $\EE$ a holomorphic rank $n$ bundle of degree $k_0$ on $\Sigma$,
\item $\rho$ a meromorphic section of $\Aut(\EE)$ of the form $F_j(z)
\diag_l(z^{k_{jl}})
G_j(z)$ near $z_j$, with $F_j, G_j$ holomorphic and invertible, and with
$\det(\rho)$ having divisor
$\sum_j\trk{  \vec{k}_j} z_j$.
\end{itemize}
\end{theorem}

\begin{proof}
Let $\Deck$ be the group of deck transformations of the covering $\Sigma'\to\Sigma$. The proofs (and therefore statements) of Lemma \ref{lemma:constructed} and Theorem \ref{thm:Simpson} work even if we replace $S^1$ invariance by $S^1\times \Deck$ invariance.  In turn, Theorem \ref{thm:main2} and its constituents Propositions \ref{prop:injective} and \ref{prop:surjective} are also true even if we consider objects that are invariant under the action of $\Deck$.  The proof is thus complete.\qed
\end{proof}

When the circle bundle is not one whose lift to a finite cover is trivial, one expects that the appropriate definition of degree and stability is obtained by taking a limit of averages over larger and larger $U$.


\def\cprime{$'$}
\providecommand{\bysame}{\leavevmode\hbox to3em{\hrulefill}\thinspace}
\providecommand{\href}[2]{#2}

\end{document}